\documentclass[preprint,11pt]{elsarticle}
\journal{ }
\usepackage{graphicx}
\usepackage{pifont,latexsym,ifthen,amsthm,rotating,calc,textcase,booktabs}
\usepackage{amsfonts,amssymb,amsbsy,amsmath}
\numberwithin{equation}{section}
\newtheorem{theorem}{Theorem}[section]
\newtheorem{lemma}[theorem]{Lemma}

\newtheorem{remark}[theorem]{Remark}
\newtheorem{proposition}[theorem]{Proposition}

\newtheorem{problem}{RHP}[section]
\newtheorem{main theorem}{Theorem}
\usepackage{natbib}
\usepackage{todonotes}
\usepackage{float}
\usepackage{subfigure}
\numberwithin{figure}{section}
\usepackage{hyperref}
\hypersetup{
	colorlinks=true, %set true if you want colored links
	linktoc=all,     %set to all if you want both sections and subsections linked
	linkcolor=blue,  %choose some color if you want links to stand out
	citecolor=blue%choose some color if you want links to stand out
}

\begin{document}
	
	% Enter full title and short title for running headers
	% \title{A Demonstration of the \LaTeXe\ Class File for the \textit{Oxford University Press Ltd Journal}}
	% \shorttitle{A Demonstration of the \textit{OUP Journal} Class File}
	
	% % Enter the publication year and the ID number of the paper
	% \volumeyear{2022}
	% \paperID{rnn999}
	
	\begin{frontmatter}
		\title{The asymptotic stability of solitons in the focusing Hirota equation on the line}
		
		% \shorttitle{A RH approach to the existence of global solutions with solitons to the FL equation on the line}
		% \author{Qiaoyuan Cheng$^1$,  Engui Fan$^1$}
		% % \address{%
			% % \affilnum{1}School of Mathematical Sciences  and Key Laboratory of Mathematics for Nonlinear Science, Fudan University, Shanghai 200433, P.R.
			% % China.}

		\author[inst2]{Ruihong Ma}

		\author[inst2]{Engui Fan$^{*,}$  }
		
		\address[inst2]{ School of Mathematical Sciences and Key Laboratory of Mathematics for Nonlinear Science, Fudan University, Shanghai,200433, China\\
			* Corresponding author and e-mail address: faneg@fudan.edu.cn  }
		% % Address / e-mail address of corresponding author
		% \correspdetails{faneg@fudan.edu.cn }
		
		% % Author name(s)
		% % Abbreviated author name for running headers
		% \abbrevauthor{Q. Cheng, and E. Fan}
		% % Abbreviated author name for first page header
		% \headabbrevauthor{Q. Cheng, and E. Fan}
		
		% Address / e-mail address of corresponding author
		% \correspdetails{corr.email@math.edu}
		
		% Received/revised/accepted dates will be entered by the publisher during production of an accepted paper. Please do not edit these placeholders for submission.
		% \received{1 Month 20XX}
		% \revised{11 Month 20XX}
		% \accepted{21 Month 20XX}
		
		% % Enter details of editor communicating this article
		% \communicated{A. Editor}

		\begin{abstract}
			In this paper, the $\overline\partial$-steepest descent method  and B\"acklund transformation are used  to
			study the asymptotic stability of solitons to  the   Cauchy problem of  focusing Hirota equation
			\begin{align}\nonumber
				& iq_t+\alpha(2|q|^2q+q_{xx})+i\beta(q_{xxx}+6|q|^2q_x)=0,  \\\nonumber
				& q(x,0)=q_0(x),
			\end{align}
			where $q_0 \in H^1(\mathbb{R})\,\cap\,L^{2,s}(\mathbb{R}),s\in(\frac{1}{2},1] .$
			It is shown that    the solution of the Cauchy problem can be expressed    in term of the solution of a Riemann-Hilbert (RH) problem.
			The solution of the RH problem     is further  decomposed      into   pure radiation solution and solitons solution
			obtained  by  using $\overline\partial$-techniques and B\"acklund transformation  respectively.
			As a directly consequence,   the asymptotic stability of solitons for the Hirota equation is obtained.
			
			% \\[3pt]
			% {\bf Keywords:}   Fokas-Lenells equation, Riemann-Hilbert problem,  Lipschitz continuous,   prior estimate,   global solutions.\\[3pt]
			% {\bf   Mathematics Subject Classification:} 35P25; 35Q51; 35Q15; 35A01; 35G25.
		\end{abstract}
		
		\begin{keyword}
			Hirota equation; Riemann-Hilbert problem; $\overline\partial$-steepest descent method,   B\"acklund transformation; asymptotic stability.
			
			\textit{Mathematics Subject Classification:}  35P25; 35Q51; 35Q15; 35A01; 35G25.
		\end{keyword}
	\end{frontmatter}
	\tableofcontents
	% \baselineskip=18pt
	% \maketitle
	
	%\tableofcontents
	%%%%%%%%%%%%%%%%%%%%%%%%%%%%%%%%%%%%%%%%%%%%
	\section{Introduction and main results}
	
	In this paper, we apply $\overline\partial$-steepest descent method and  B\"acklund transformation (BT) to obtain the asymptotic stability of solitons to the Cauchy problem of the Hirota equation,
	\begin{align}
		&iq_t+\alpha(2|q|^2q+q_{xx})+i\beta(q_{xxx}+6|q|^2q_x)=0, \label{eq:1.1}  \\
		&q(x,0)=q_0(x), \  \ (x,t)\in\mathbb{R}\times\mathbb{R}^+, \label{eq:1.2}
	\end{align}
	where $q_0\in H^1(\mathbb{R})\,\cap\,L^{2,s}(\mathbb{R}), s\in(\frac{1}{2},1]$,   the real constants $\alpha$ and $\beta$ stand for the second-order and third-order dispersions.
	The Hirota equation (\ref{eq:1.1}) is a typically mathematical physical model, which encompasses the well-known NLS equation and  derivative NLS equation \cite{1,2}.
	It is a more accurate approximation than the NLS equation in describing wave propagation in the ocean and optical fiber \cite{3,4,5}.
	
	In recent years, much work has been  done to study the various mathematical properties of the Hirota equation.
	For example, exact solutions  such as  multisoliton solutions, breather solutions, rational solutions and rogue wave solutions for  the Hirota equation
	was widely studied   \cite{8,9,10,11}.  The  $N$-soliton solutions for the  Hirota equation with non-zero boundary condition
	were constructed by using the  Riemann-Hilbert (RH)  method  \cite{14}.
	The   initial boundary value problem  for the Hirota Equation on the half line  was analyzed by using the Fokas unified method \cite{15}.
	The long time asymptotics for the Hirota equation with decaying initial data was investigated  via the nonlinear steepest descent method   \cite{16}.
	We further  found    the Painlev\'e  asymptotics for the  Hirota equation with Schwartz Cauchy
	data in the transition region \cite{XF}.
	It was shown that the Cauchy problem for the Hirota equation is  globally   well-posed   in the  space $H^s (\mathbb{R}), s\geq 1$ \cite{18}, and admits the $L^2$-conservation law $||q(t)||_{L^2}=||q_0||_{L^2}, \ t\in\mathbb{R}$.
	Further the Cauchy   problem for the Hirota equation is    well-posed   in the  space $H^s (\mathbb{R}), s>1/2$ \cite{guo2}.
	The   orbital stability of solitons for the Hirota equation in Sobolev space $H^1(\mathbb{R})$ was shown  \cite{17}.
	
	In this paper, we use the $\overline\partial$-steepest descent method and B\"acklund transformation
	to acquire the  asymptotic stability of  solitons for the  Hirota equation in  Sobolev space  $H^1(\mathbb{R})\,\cap\,L^{2,s}(\mathbb{R})$.
	The asymptotic analysis of solitons for focusing Hirota equation is necessarily more detailed than the defocusing case due to the presence of solitons which correspond to discrete spectrum of the non self-adjoint
	operator associated with (\ref{eq:1.1}).

	%We use its equivalent integral formulation
	%\begin{equation}
	%q(t)=W(t)q_0+i\int_0^tW(t-t')(6i\beta|q|^2\partial_xq+2\alpha|q|^2q)(t')dt'.
	%\end{equation}
	%where
	%$$W(t)=\mathcal{F}_x^{-1}e^{-it(\alpha\xi^2-\beta\xi^3)}\mathcal{F}_x.$$
	%is the unitary operator associated to the linear equation,

	We are interested in asymptotic stability of solitons in the Hirota equation  (\ref{eq:1.1}) given by the explicit expressions \cite{19}
	\begin{align}\label{eq:1.3}
		q_{(\eta,\xi,\gamma,\beta,\alpha)}(x,t)&=2\eta  e^{2i(-\xi x-4\beta\xi^3t-2\alpha\xi^2t+12\beta\xi\eta^2t+2\alpha\eta^2t)+i\gamma}\nonumber\\
		&\times\mbox{sech}(-2\eta x-24\beta\eta\xi^2t+8\beta\eta^3t-8\alpha\eta\xi t),
	\end{align}
	where $\lambda=\xi+i\eta$ is a complex nonzero parameter that determines by solitons. We consider here  the question of their asymptotic stability that is when $q_0$ is close to $q_{(\eta,\xi,\gamma)}$ for
	a particular $(\eta,\xi,\gamma)$. Our   principal result  is  now stated as follows.

	\begin{theorem}  \label{th1} Let $q_0\in H^1(\mathbb{R})\,\cap\,L^{2,s}(\mathbb{R})$ for fixed $s\in (1/2,1]$ and a soliton $q_{(\eta,\xi,\gamma)}(0,
		x)$ of the Hirota equation. Then, there exist positive constants $\epsilon_0=\epsilon_0(\eta_0,\xi_0)$, $T=T(\eta_0,\xi_0)$, $C=C(\eta_0,\xi_0)$ and if
		\begin{equation}\label{eq:1.4}
			\epsilon:=||q_{(\eta_0,\xi_0,\gamma_0)}(\cdot,0)-q_0||_{H^1(\mathbb{R})\,\cap\,L^{2,s}(\mathbb{R})}<\epsilon_0.
		\end{equation}
		then there exist a state solution  $q_{(\eta_1,\xi_1,\gamma_1)}(x,t)$ such that for the solution of the Cauchy problem (\ref{eq:1.1})-(\ref{eq:1.2}), we have
		\begin{equation}
			|(\eta_1,\xi_1,\gamma_1,x)-(\eta_0,\xi_0,\gamma_0,,x_0)|<C\epsilon,
		\end{equation}
		and for all $|t|\geq T$,
		\begin{equation}
			||q(\cdot,t)-q_{(\eta_1,\xi_1,\gamma_1)}(\cdot,t)||_{L^{\infty}}<C\epsilon|t|^{-\frac{1}{2}}.
		\end{equation}
	\end{theorem}
	%In general the two ground states $q_{(\eta_1,\xi,\gamma_{\pm},\beta_1,\alpha)}(t,x-x_{\pm})$ are distinct in \textcolor{blue}{Lemma 4.4}.
	
	Organization   of the paper is as follows.    In Section \ref{sec2},  we describe the inverse scattering transform   to formulate  the Cauchy problem (\ref{eq:1.1})-(\ref{eq:1.2})
 into     a RH problem \ref{RH21} (see below).   In Section \ref{sec3},
	with   the $\overline\partial$ analysis,  we obtain a   RH problem with pure radiation  solution.
 Removing this component of the solution results in a $\overline\partial$ problem which is  analyzed in Subsection \ref{dbar}. In Section  \ref{sec4}, a  B\"acklund transformation is constructed to establish  the relation  between soliton and soliton free solutions.
 We estimates the norm of the transformation in Lemma \ref{le43}, and further give  the proof of Theorem \ref{th1}.

	\section{Direct and inverse  scattering transforms} \label{sec2}

	In this section, we will analyze the spectral problem of the focusing Hirota equation
	(\ref{eq:1.1}) to obtain Jost functions and scattering matrix related to initial value $q_0$.
	
	\subsection{Notations}
	
	With regard to complex variables, given a variable $z$ or a function  $f(z)$, 	 	we denote by $z^*$ and $f^*(z)$ their respective complex conjugates;         The symbol $\overline\partial$ denoted the derivative with respect to $z^*$, i.e. if $z=x+iy$, then
	$$\overline\partial f=\frac{1}{2}(f_x+if_y).$$
	We  introduce  the Japanese bracket $\langle x \rangle=\sqrt{1+x^2}$ and  the following  normed spaces:
	A weighted   space  $L^{p,s}(\mathbb{R})$ defined with the norm
	$$\|q\|_{L^{p,s}(\mathbb{R})} := \|\langle x \rangle^s q\|_{L^{p}(\mathbb{R}) };$$
	
	A Sobolev space $W^{k,p}(\mathbb{R})$ defined with the norm
	$\|q\|_{W^{k,p}(\mathbb{R})}:=\sum_{j=0}^k \|\partial^jq\|_{L^p(\mathbb{R})},$
	where $\partial^jq$ is the $j^{th}$ weak derivative of $q$;
	
	A Sobolev space $H^k(\mathbb{R})$ defined with the norm
	$\|q\|_{H^k(\mathbb{R})}:=\|\langle x \rangle^k\mathcal{F} (q)\|_{L^2(\mathbb{R})},$
	where $\mathcal{F} (q)$ is the Fourier transform of $q$;
	
	Recall that $L^{2,s}(\mathbb{R})$ is embedded into $L^1(\mathbb{R})$ for any $s>\frac{1}{2}$. Based on this fact  we consider the potential function $q(x)\in L^1(\mathbb{R})$ for simplicity.

	\subsection{Jost functions and scattering data}
	
	The focusing Hirota equation (\ref{eq:1.1}) is just a compatibility condition  of the system of linear partial differential equation
	\begin{align}\label{eq:2.1}
		&\Psi_x=M(x,t;z)\Psi,\quad M(x,t;z)=-iz\sigma_3+Q,\\\label{eq:2.2}
		&\Psi_t=N(x,t;z)\Psi,\quad\,N(x,t;z)=-i(4\beta z^3+2\alpha z^2)\sigma_3+V,
	\end{align}
	where $\Psi=\Psi(x,t;z)$ is a 2$\times$2 matrix-valued eigenfunction, $z\in\mathbb{C}$ is a spectral parameter, and
	\begin{equation}\nonumber
		Q=\begin{pmatrix}
			0&q(x,t)\\
			-q^*(x,t)&0
		\end{pmatrix},  \ \ \sigma_3=\begin{pmatrix}
			1&0\\
			0&-1
		\end{pmatrix},
	\end{equation}
	\begin{equation}\nonumber
		V=4\beta z^2Q+zV_1+V_0,\quad V_1=2\alpha Q-2i\beta(Q_x+Q^2)\sigma_3,
	\end{equation}
	\begin{equation}\nonumber
		V_0=-i\alpha(Q_x+Q^2+q_0^2)\sigma_3+\beta[Q_x,Q]+\beta(2Q^3-Q_{xx}).
	\end{equation}
	The Lax pair  (\ref{eq:2.1})-(\ref{eq:2.2}) admits a Jost function with asymptotics
	\begin{equation}
		\Psi(x,t;z)\thicksim e^{-it\theta(z)\sigma_3},\quad |x|\to\infty,
	\end{equation}
	where
	\begin{equation}\label{eq:2.4}
		\theta=\theta(x,t;z)=z\frac{x}{t}+2\alpha z^2+4\beta z^3.
	\end{equation}
	Define  the eigenfunctions $\Phi(x,t;z)$
	\begin{equation}
		\Phi(x,t;z)=\Psi(x,t;z)e^{it\theta(z)\sigma_3},
	\end{equation}
	then $ \Phi(x,t;z)\to\mathnormal{I},\ \ |x|\to\infty $, and  we obtain a equivalent Lax pair
	\begin{align}
		&\Phi_x+iz[\sigma_3,\Phi]=Q\Phi,\\
		&\Phi_t+i(2\alpha z^2+4\beta z^3)[\sigma_3,\Phi]=V\Phi,
	\end{align}
	whose    solutions can be expressed as Volterra type integrals
	\begin{align}
		&\Phi_-(x,t;z)=\mathnormal{I}+\int_{-\infty}^xe^{-iz(x-y)\hat{\sigma_3}}Q(y,t)\Phi_-(y,t;z)dy,\\
		&\Phi_+(x,t;z)=\mathnormal{I}-\int_x^{+\infty}e^{-iz(x-y)\hat{\sigma_3}}Q(y,t)\Phi_+(y,t;z)dy,
	\end{align}
	where $e^{\hat{\sigma}_3 }A=e^{\sigma_3}Ae^{-\sigma_3}$ with $A$ is $2\times2$ matrix.

Let $\Phi_{\pm,j}(x,t;z)(j=1,2)$ denote the $j$-th column of $\Phi_{\pm}(x,t;z)$,  then it can be shown that  $\Phi_{-,1}(x,t;z)$ and $\Phi_{+,2}(x,t;z)$ are analytic in $C^+=\{z\,|\, \mbox{Im}z>0\}$ and continuous in $C^+\cup\,\mathbb{R}=\{z\,|\,\mbox{Im}z\ge0\}$, $\Phi_{+,1}(x,t;z)$ and $\Phi_{-,2}(x,t;z)$ are analytic in $C^-=\{z\,|\,\mbox{Im}z<0\}$ and continuous in $C^-\cup\mathbb{R}=\{z\,|\,\mbox{Im}z\leq0\}$.
	
	Since $\Phi_+(x,t;z)e^{-it\theta(z)\sigma_3}$ and  $\Phi_-(x,t;z)e^{-it\theta(z)\sigma_3}$ are fundamental matrix solution of the Lax pair (\ref{eq:2.1})-(\ref{eq:2.2}), there exists a scattering matrix   $S(z)=(s_{ij}(z))_{2\times2}$ satisfying
	\begin{equation}
		\Phi_-(x,t;z)=\Phi_+(x,t;z) e^{-it\theta(z)\hat\sigma_3} S(z),\quad z\in\mathbb{R}.
	\end{equation}
	The scattering coefficients can be expressed by using Wronskians
	\begin{align}\nonumber
		&s_{11}(z)=\det(\Phi_{-,1}(x,t;z)\quad \Phi_{+,2}(x,t;z)),\\\nonumber &s_{12}(z)=e^{-2it\theta(z) }\det(\Phi_{+,2}(x,t;z)\quad \Phi_{-,2}(x,t;z)),\\\nonumber
		&s_{21}(z)=e^{ 2it\theta(z } \det(\Phi_{+,1}(x,t;z)\quad \Phi_{-,1}(x,t;z)),\\\nonumber
		& s_{22}(z)=\det(\Phi_{-,2}(x,t;z)\quad \Phi_{+,1}(x,t;z)),\nonumber
	\end{align}
	where scattering data $s_{11}(z)$ is analytic in $C^+$ and continuous on $C^+\cup\mathbb{R}$ and $s_{22}(z)$ is analytic in $C^-$ and continuous on $C^-\cup\mathbb{R}$. In addition, $s_{12}(z)$ and $s_{21}(z)$ are continuous in $\mathbb{R}$. It follows the symmetric relation
	$$\Phi(x,t;z)=\sigma_2\Phi^*(x,t;z^*)\sigma_2,\quad S(z)=\sigma_2S^*(z^*)\sigma_2.$$
	We can obtain $s_{11}(z)=s_{22}^*(z^*)$ and $s_{12}(z)=-s_{21}^*(z^*)$.
	Define   reflection coefficient
	\begin{equation}\label{r}
		r(z):=\frac{s_{21}(z)}{s_{11}(z)},\,\,z\in\mathbb{R},
	\end{equation}
	then it can be shown that \cite{Beals1984,Cuccagna2013}.

	%\noindent\textbf{\expandafter{Lemma 2.1} }
	\begin{lemma}\label{le21}
		\label{la21} There exists an open dense set $\mathcal{G}\subset L^1(\mathbb{R})$ such that, for $q\in\mathcal{G}$, the scattering function $s_{11}(z)$ has at most a finite number of zeros forming a set $\mathcal{Z}_+=\{z_1,\dots,z_N\}$ in $\mathbb{C}_+$, with $s_{11}(z)\ne0$ for all $z\in\mathbb{R}$ and $s_{11}'(z_k)\ne0$ for all $k$. Denote $\mathcal{Z}_-=\{z^*_1,\dots,z^*_N\}$ and the cardinality $q\to\sharp\mathcal{Z}$ is locally constant near $q$ in $\mathcal{G}$ and the map $\mathcal{G}\ni q\to(z_1,\dots,z_N)\in\mathbb{C}_+^N$ is locally Lipschitz.
	\end{lemma}
	We denote by $\mathcal{G}_N$ the open subset of $\mathcal{G}$ formed by the elements such that $\sharp\mathcal{Z}=N$. We call the potential in $\mathcal{G}$ \emph{generic}. $\mathcal{G}_1$ simply contains the soliton (\ref{eq:1.3}).  It is shown that   the solitons  under small $L^1$-perturbations are still  in $\mathcal{G}_1$   \cite{DP} .
	
	%\noindent\textbf{\expandafter{Lemma 2.2} }
	\begin{lemma} \label{le22}Let $s\in(\frac{1}{2},1]$. For $q\in H^1(\mathbb{R})\,\cap\,L^{2,s}(\mathbb{R})\cap\,\mathcal{G}$ we have $r\in H^s(\mathbb{R})\cap L^{2,1}(\mathbb{R})$. Furthermore, the map $ H^1(\mathbb{R})\,\cap\,L^{2,s}(\mathbb{R})\cap\mathcal{G}\ni q\to r\in H^s(\mathbb{R})\cap L^{2,1}(\mathbb{R})$ is locally Lipschitz.
	\end{lemma}
	\begin{proof}For any fixed $\kappa_0>0$ there exists a positive constant $C$ such that if $||q||_{H^1(\mathbb{R})\,\cap\,L^{2,s}(\mathbb{R})}\leq\kappa_0$
		satisfying
		\begin{align}\nonumber
			&||\Phi_{-,j}(x,\cdot )-e_j||_{H^s_z(\mathbb{R})}\leq C||q||_{L^{2,s}(\mathbb{R})},\quad \forall\,\,x\leq0,j=1,2,\\\nonumber
			&||\Phi_{+,j}(x,\cdot)-e_j||_{H^s_z(\mathbb{R})}\leq C||q||_{L^{2,s}(\mathbb{R})},\quad
			\forall\,\,x\ge0,j=1,2,\\\nonumber
			&||\Phi_{-,j}(x,\cdot)-e_j||_{L_z^{2,1}(\mathbb{R})}\leq C||q||_{H^1(\mathbb{R})},\quad
			\forall\,\,x\leq0,j=1,2,\\\label{eq:2.12}
			&||\Phi_{+,j}(x,\cdot)-e_j||_{L_z^{2,1}(\mathbb{R})}\leq C||q||_{H^1(\mathbb{R})},\quad
			\forall\,\,x\ge0,j=1,2.
		\end{align}
		We now consider the case $j=1$ and the minus sign only to prove (\ref{eq:2.12}). We based on the fact that if there is $s\in(0,1]$ such that for an $f\in L^2(\mathbb{R})$ we have
		\begin{equation}\label{eq:2.15}
			||f(\cdot+h)-f(\cdot)||_{L^2(\mathbb{R})}\leq C|h|^s,\qquad\forall h\in\mathbb{R},
		\end{equation}
		then $f\in H^s(\mathbb{R})=\dot{H}^s(\mathbb{R})\cap L^2({\mathbb{R})}$ and there is a positive constant $c$ independent of $f$ such that $||f||_{\dot{H}^s(\mathbb{R})}\leq cC,$ where  $\dot{H}^s(\mathbb{R})$ defined with $||q||_{\dot{H}^s(\mathbb{R})}:=|||x|^s\mathcal{F}{(q)}||_{L^2(\mathbb{R})}$.
		\\
		Define
		\begin{equation}
			Kf(x,z):=\int_{-\infty}^x\begin{pmatrix}
				e^{-2i(x-y)z}&0\\
				0&e^{2i(x-y)z}\end{pmatrix}Q(q(y))f(y,z)dy.
		\end{equation}
		According to the \cite{20,21} we obtain
		$$||Ke_2||_{L_x^\infty(\mathbb{R},L^2_z(\mathbb{R}))}\leq ||q||_{L^2(\mathbb{R})},$$
		hence we acquire that
		\begin{equation}\nonumber
			||(1-K)^{-1}||_{L_x^\infty((-\infty,x_0),L_z^2(\mathbb{R})\to L_x^\infty(-\infty,x_0),L_z^2(\mathbb{R}))}\leq e^{||q||_{L^1}},\quad \forall x_0\leq+\infty.
		\end{equation}
		%	Use the plancherel formula we have
		%	$$||\Phi_{-,j}(x,z)-e_j||_{L_z^{2,2}(\mathbb{R})}=|||z|^2|\Phi_{-,j}(x,z)-e_j|||_{L^2(\mathbb{R})}=
		%	||\Phi_{-,j}(x,z)-e_j||_{H^s_z(\mathbb{R})}.\,\, \forall x\leq0,j=1,2.$$
		We consider the case for $x\leq0$ we obtain
		\begin{align}\nonumber
			&(1-K)(\Phi_{-,1}(x,z)-e_1)=Ke_1=\int_{-\infty}^x\begin{pmatrix}
				0\\
				-q^*e^{2i(x-y)z}
			\end{pmatrix}dy,
		\end{align}
		and
		\begin{align}\nonumber
			||\Phi_{-,1}(x,z)-e_1||_{L^2_z(\mathbb{R})}&\leq e^{||q||_{L^1}}\Big|\Big|\int_{-\infty}^xe^{-2i(x-y)z}q^*(y)dy\Big|\Big|_{L_z^2(\mathbb{R})}\\\nonumber
			&\leq e^{||q||_{L^1}}\Big(\int_{-\infty}^x\langle y\rangle^{2s}|q(y)|^2dy\Big)^{1/2}\langle x\rangle^{-s}\\
			&\leq e^{||q||_{L^1}}||q||_{L^{2,s}(\mathbb{R})}\langle x\rangle^{-s},
		\end{align}
		and
		\begin{align}\nonumber
			||\Phi_{-,1}(x,z)-e_1||_{L_z^{2,1}(\mathbb{R})}&\leq e^{||q||_{L^1}}\Big|\Big|\int_{-\infty}^x ze^{2i(x-y)z}q^*(y)dy\Big|\Big|_{L_z^2(\mathbb{R})}\\\nonumber
			&\leq e^{||q||_{L^1}}\Big(\int_{-\infty}^x ze^{2i(x-y)z}q^*(y)+q^*(y)dy\Big|\Big|_{L_z^2(\mathbb{R})}\\
			&\leq e^{||q||_{L^1}}\sqrt{\pi}||q||_{H^{1}(\mathbb{R})}.
		\end{align}
		Next we define $\mathcal{N}(x,z):=\Phi_{-,1}(x,z+h)-\Phi_{-,1}(x,z)$ for $h\in\mathbb{R}$
		to estimate the form (\ref{eq:2.15}) by calculating the inequality  $C\lesssim ||q||_{L^{2,s}(\mathbb{R})}$, we have
		\begin{align}\nonumber
			(1-K)&\mathcal{N}(x,z)=\int_{-\infty}^x\begin{pmatrix}
				e^{-2i(x-y)(z+h)}-e^{-2i(x-y)z}&0\\
				0&e^{2i(x-y)(z+h)}-e^{2i(x-y)z}
			\end{pmatrix}\\\label{eq:2.19}
			&Q(q(y))(\Phi_{-,1}(y,z)-e_1)dy+\int_{-\infty}^x\begin{pmatrix}
				0\\
				(e^{2i(x-y)(z+h)}-e^{2i(x-y)z})q^*(y)
			\end{pmatrix}dy.
		\end{align}
		Using the Fourier transform $\mathcal{F}$ we have for $x\leq0$ the second term r.h.s (\ref{eq:2.19})
		\begin{align}\nonumber
			 \bigg\|\int_{-\infty}^x&(e^{2i(x-y)(z+h)}-e^{2i(x-y)z})q^*(y)dy \bigg\|_{L_z^2}\\\nonumber
			&=\|\mathcal{F}^*[q(\cdot+x)\chi_{\mathbb{R}_-}](z+h)-\mathcal{F}^*[q(\cdot+x)\chi_{\mathbb{R}_-}](z)||_{L_z^2}\\\nonumber
			&\leq C||\mathcal{F}^*[q(\cdot+x)\chi_{\mathbb{R}_-}](z)||_{H_z^s(\mathbb{R})}|h|^s\\
			&=||q(y+x)||_{L_y^{2,s}(\mathbb{R}_-)}|h|^s\leq ||q||_{L^{2,s}(\mathbb{R})}|h|^s,
		\end{align}
		and the first term r.h.s (\ref{eq:2.19}) is
		\begin{align}\nonumber
			\int_{-\infty}^x&\begin{pmatrix}
				e^{-2i(x-y)(z+h)}-e^{-2i(x-y)z}&0\\
				0&e^{2i(x-y)(z+h)}-e^{2i(x-y)z}
			\end{pmatrix}\\\nonumber
			&Q(q(y))(\Phi_{-,1}(y,z)-e_1)dy\\\nonumber
			&\leq 2^{(1-s)}|h|^s\int_{-\infty}^x|y|^s|q(y)||\Phi_{-,1}(y,z)-e_1||_{L_z^2(\mathbb{R})}dy\\\nonumber
			&\leq 2^{(1-s)}|h|^se^{||q||_{L^1}}||q||_{L^{2,s}(\mathbb{R})}\int_{-\infty}^x|y|^s\langle y\rangle^{-s}|q(y)| dy\\
			&\leq 2^{(1-s)}|h|^se^{||q||_{L_1}}||q||_{L^{2,s}(\mathbb{R})}||q||_{L^1}.
		\end{align}
		Then, we obtain
$$||\Phi_{-,1}(x,z+h)-\Phi_{-,1}(x,z)||_{L_z^2(\mathbb{R})}\leq C|h|^s||q||_{L^{2,s}(\mathbb{R})},  \ x\leq0,$$
 where $C$ is a fixed constant for $||q||_{H^1(\mathbb{R})\,\cap\,L^{2,s}(\mathbb{R})}\leq\kappa_0$, for a preassigned bound $\kappa_0$. This implies that for all $x\leq0$ we have $$||m_{-,1}(x,z)-e_1||_{\dot{H}_z^s(\mathbb{R})}\leq C||q||_{L^{2,s}(\mathbb{R})},$$
  for  some positive constant $C$ and above the assume have been proved.  The other cases similar.
		Now we conclude that (\ref{eq:2.12}) be established. Then $s_{12}(z)\in H^s(\mathbb{R})\,\cap\,L^{2,1}(\mathbb{R})$ because
		\begin{align}\nonumber
			s_{12}(z)&=\det(\Phi_{+,2}(0,z),\Phi_{-,2}(0,z))\\\nonumber
			&=\det(\Phi_{+,2}(0,z)-e_2,\Phi_{-,2}(0,z)-e_2)\\
			&+\det(\Phi_{+,2}(0,z)-e_2,e_2)+\det(e_2,\Phi_{-,2}(0,z)-e_2),
		\end{align}
		and we notice that $H^s(\mathbb{R})$ is a Banach algebra with respect to pointwise multipilication for any $s>\frac{1}{2}$. Similarly, $(s_{11}-1)\in H^s(\mathbb{R})\,\cap\,L^{2,1}(\mathbb{R})$ because
		\begin{align}\nonumber
			s_{11}(z)&=\det(\Phi_{-,1}(0,z),\Phi_{+,2}(0,z))\\\nonumber
			&=\det(\Phi_{-,1}(0,z)-e_1,\Phi_{+,2}(0,z))+\det(e_1,\Phi_{+,2}(0,z)-e_2)+\det(e_1,e_2)\\\nonumber
			&=1+\det(\Phi_{-,1}(0,z)-e_1,\Phi_{+,2}-e_2)\\
			&+\det(e_1,\Phi_{+,2}-e_2)+\det(\Phi_{-,1}(0,z)-e_1,e_2).
		\end{align}
		We conclude that if $q\in H^1(\mathbb{R})\,\cap\,L^{2,s}(\mathbb{R})\cap\,\mathcal{G}$ then $r\in H^s(\mathbb{R})\cap L^{2,1}(\mathbb{R})$ and this shows that we accquire a map $H^1(\mathbb{R})\,\cap\,L^{2,s}(\mathbb{R})\cap\,\mathcal{G}\ni q\to r\in H^s(\mathbb{R})\cap L^{2,1}(\mathbb{R})$.The locally Lipschitz of the map  $H^1(\mathbb{R})\,\cap\,L^{2,s}(\mathbb{R})\cap\,\mathcal{G}\ni q\to r\in H^s(\mathbb{R})\cap L^{2,1}(\mathbb{R})$ we skip.\\
	\end{proof}

	\subsection{A Riemann-Hilbert problem}
	Define the collection what we need in the space
	\begin{align}\nonumber
		\mathcal{S}(s,n):=&\{r(z)\in H^{s}(\mathbb{R})\cap L^{2,1}(\mathbb{R}),\quad (z_1,\dots,z_n)\in\mathbb{C}_+^n,\\\label{eq:2.21}
		&(c_1,\dots,c_n)\in\mathbb{C}_*^n:=\mathbb{C}^n\setminus\{0\}\},
	\end{align}
	is called the scattering data for for initial data $q_0$ and the map $\mathcal{P}:q_0\mapsto\mathcal{S}$ is called the (forward) scattering map.
	\begin{proposition}\label{p23}
	 If $r(z)\in H^s(\mathbb{R})\,\cap\,L^{2,1}(\mathbb{R})$,
	 then for every $t\in\mathbb{R}$, we have $r(z,t)=r(z)e^{i(4\beta z^3+2\alpha z^2)t}\in H^{s}(\mathbb{R})\cap L^{2,1}(\mathbb{R})$.
	 	\end{proposition}
	 \begin{proof}By (\ref{r}), we obtain
	 	\begin{align*}
	 	&	\|r(\cdot,t)\|_{L^{2,1}(\mathbb{R})}=\|z r(z)e^{i(4\beta z^3+2\alpha z^2)t}\|_{L^2(\mathbb{R})}\\
	 	&\qquad\qquad\qquad=\| z r(z)\|_{L^2(\mathbb{R})}=||r(z)||_{L^{2,1}(\mathbb{R})},\\
	 and
	 \\
	 	&\|r(\cdot,t)\|_{H^s(\mathbb{R})}=\| z^s\mathcal{F}(r(z)e^{i(4\beta z^3+2\alpha z^2)t})\|_{L^2(\mathbb{R})}\\
	 	&\qquad\qquad\qquad=(2\pi)^{-1}\| z^s\mathcal{F}(r(z))*\mathcal{F}(e^{i(4\beta z^3+2\alpha z^2)t})||_{L^2(\mathbb{R})}\\
	 	&\qquad\qquad\qquad=(2\pi)^{-1}\|r(z)\|_{H^{s}(\mathbb{R})}.
	 \end{align*}
	 Therefore, we completed the proof of Proposition \ref{p23}.
	 \end{proof}

	The essential fact of integrability is that if the potential $q_0$ evolves according to (\ref{eq:1.1}) the the time evolution of the scattering data $\mathcal{D}$ is trivial and the collection
	\begin{align}\nonumber
		\mathcal{D}(t):=&\{r(z,t) \in H^{s}(\mathbb{R})\cap L^{2,1}(\mathbb{R}),\quad (z_1,\dots,z_n)\in\mathbb{C}_+^n,\\
		& (c_1e^{i(4\beta z_1^3+2\alpha z_1^2)t},\dots,c_ne^{i(4\beta z_n^3+2\alpha z_n^2)t})\in\mathbb{C}^n\setminus\{0\}\}.
	\end{align}
	The inverse scattering map $\mathcal{P}^{-1}:\mathcal{D}(r)\mapsto q(x,t)$ seeks to recover the solution of (\ref{eq:1.1}) from its scattering data.
	
	We construct the function
	\begin{equation}
		M(z):=\begin{cases}
			\big(\frac{\Phi_{-,1}(x,t;z)}{s_{11}(z)},\Phi_{+,2}(x,t;z)\big),\quad z\in\mathbb{C}^+,
			\\
			\\
			(\Phi_{+,1}(x,t;z),\frac{\Phi_{-,2}(x,t;z)}{s_{22}(z)}\big),\quad\, z\in\mathbb{C}^-,
		\end{cases}
	\end{equation}
	which satisfies  the following RH problem.
	
	\begin{problem}\label{RH21}
		Find an analytic function $M(z):\mathbb{C}\setminus(\mathbb{R}\,\cup\,\mathcal{Z}\,\cup\,\mathcal{Z}^*)\to SL_2(\mathbb{C})$ with the following properties
		
		\begin{enumerate}
			\item $ M(z)=I+\mathcal{O}(z^{-1})$ as $z\to\infty$.
			\item $M(z)$ takes continuous boundary values $M_{\pm}(z)$ which satisfy the jump relation
			$$M_+(z)=M_-(z)V(z),$$
			where
			\begin{equation}
				V(z)=\begin{pmatrix}
					1+|r(z)|^2&r^*(z)e^{-2it\theta(z)}\\
					r(z)e^{2it\theta(z)}&1
				\end{pmatrix}.
			\end{equation}
			\item $M(z)$ has simple poles at each $\mathcal{Z}:=\mathcal{Z}_+\cup\mathcal{Z}_- $ at which
			\begin{align}
				&\underset{z=z_k}{\mbox{Res}} M(z)=\underset{z\to z_k}{\mbox{lim}}M(z)\begin{pmatrix}
					0&0\\
					c_ke^{-2it\theta(z)}&0
				\end{pmatrix},\\
				&\underset{z=z_k^*}{\mbox{Res}} M(z)=\underset{z\to z_k^*}{\mbox{lim}}M(z)\begin{pmatrix}
					0&-c_k^*e^{2it\theta(z^*)}\\
					0&0
				\end{pmatrix}.
		\end{align}\end{enumerate}
	\end{problem}
	
	It's a simple consequence of Liouville's theorem that if a solution exist, it is unique. The existence of solutions of RHP 2.1 for any $(x,t)\in\mathbb{R}\times\mathbb{R}$ follows by means of Zhou's vanishing lemma argument \cite{22}. Expanding this solution as
	$$M(z)=I+\frac{M^{(1)}(x,t)}{z}+\mathcal{O}(z^{-2}), \ z\to\infty, $$
	and one can find the solution
	\begin{equation}
		M(x,z)=I-\underset{\zeta\in\mathcal{Z}}{\sum}\frac{M_x(\zeta)V(\zeta)}{\zeta-z}+\frac{1}{2\pi i}\int_\mathbb{R}\frac{M_x(V(\zeta)-I)}{\zeta-z}d\zeta,
	\end{equation}
	where $M_x(z)$ is defined for $z\in\mathbb{R}\cup\mathcal{Z}$ in the space $M_{2\times2}(\mathbb{C})$ of complex $2\times2$ matrices and satisfies system (\ref{eq:2.29}) and (\ref{eq:2.30}) written below
	and it follows that the solution of (\ref{eq:1.1}) is given by
	\begin{equation}\label{eq:2.28}
		q(x,t)=2i\underset{z\to\infty}{\mbox{lim}}(zM(x,t;z))_{12}=2i(M^{(1)})_{12}.
	\end{equation}

	\begin{lemma}\label{le23} Fix $s\in(\frac{1}{2},1]$ and suppose that $r\in H^s(\mathbb{R})\cap L^{2,1}(\mathbb{R})$. Then, for any $x\in\mathbb{R}$ there exists and unique a solution $M_x:\mathbb{R}\cup\mathcal{Z}\to\mbox{M}_{2\times2}(\mathbb{C})$ of the following system of integral and algebraic equation:
		\begin{equation}\label{eq:2.29}
			M_x(z)=I-\underset{\zeta\in\mathcal{Z}}{\sum}\frac{M_x(\zeta)V(\zeta)}{\zeta-z}+\underset{\epsilon\to0}{\mbox{lim}}\int_{\mathbb{R}}\frac{M_x(\zeta)(V(\zeta)-I)}{\zeta-(z-i\epsilon)}d\zeta,\quad z\in\mathbb{R},
		\end{equation}
		and
		\begin{equation}\label{eq:2.30}
			M_x(z)=I-\underset{\zeta\in\mathcal{Z}\setminus\{z\}}{\sum}\frac{M_x(\zeta)V(\zeta)}{\zeta-z}+\frac{1}{2\pi i}\int_{\mathbb{R}}\frac{M_x(\zeta)(V(\zeta)-I)}{\zeta-z}d\zeta,\quad z\in\mathcal{Z},
		\end{equation}
		such that $(M_x(z)-I)\in L_z^2(\mathbb{R})$.
	\end{lemma}
	\begin{proof}
	\end{proof}
	The result of the  Lemma \ref{le23} implies that the map $\mathcal{G}_N\cap H^1(\mathbb{R})\,\cap\,L^{2,s}(\mathbb{R})\to\mathcal{S}(s,n)$ is one-to-one. The result is due to Zhou \cite{23}, we just the completeness of it.
Now we just consider the case of pure radiation solutions of the Hirota equation with $N=0$. And the Lemma \ref{le24} establish the fact that the map $\mathcal{G}_0\cap H^1(\mathbb{R})\,\cap\,L^{2,s}(\mathbb{R})\to\mathcal{S}(s,0)$ is only one-to-one but also onto.
	\begin{lemma}\label{le24}
		Let $r\in H^s(\mathbb{R})\cap L^{2,1}(\mathbb{R})$, $\mathcal{Z}=\varnothing$ and consider the potential $q$ defined by the reconstructing formula (\ref{eq:2.28}). Then $q\in H^1(\mathbb{R})\,\cap\,L^{2,s}(\mathbb{R})$. Furthermore, for any positive $\kappa_0$, there is a constant $C$ such that for $||r||_{L^\infty(\mathbb{R})}\leq\kappa_0$, we have
$$||q||_{H^1(\mathbb{R})\,\cap\,L^{2,s}(\mathbb{R})}\leq C||r||_{H^{s}(\mathbb{R})\,\cap\, L^{2,1}(\mathbb{R})}.$$
	\end{lemma}
	\begin{proof} We now construct the inverse scattering transform from $H^{s}(\mathbb{R})\,\cap\, L^{2,1}(\mathbb{R})$ to $H^1(\mathbb{R})\,\cap\,L^{2,s}(\mathbb{R})$. Let $r\in H^{s}(\mathbb{R})\,\cap\, L^{2,1}(\mathbb{R})$ be given. The inverse problem is formulated as RH Problem \ref{RH21}.

		We factorize the jump matrix  $V(z)$ on the line $ z\in(-\infty,z_1)\cup(z_2,+\infty)$ where
		\begin{equation}
			V(z)=V_-^{-1}(z)V_+(z)=
			\begin{pmatrix}
				1&r^*(z)e^{-2it\theta(z)}\\
				0&1
			\end{pmatrix}\begin{pmatrix}
				1&0\\
				r(z)e^{2it\theta}&1\end{pmatrix}.\quad\end{equation}
		This problem is adapted for studying the decay behavior of $q$ as $x\to-\infty$. Next, we prove the decay behavior of $q$ as $x\to-\infty$. The decay behavior of $q$ as $x\to\infty$ can be obtained in a similar manner.	
			
		Let
$$C_{w_x}h:=C^+(hw_{x-})+C^-(hw_{x+}),$$
where   $w_{x\pm}:=\pm(V_{\pm}-I).$
 Then we consider a function $\mu\in I+L^2(\mathbb{R})$ such that
		\begin{equation}
			(I-C_w)(\mu_x)(z)=I.
		\end{equation}
		We acquire the potential $q(x,t)$ in (\ref{eq:2.28}) by $w(\zeta):=V_+(\zeta)-V_-(\zeta)$ in the case of $c_j=0$ which can be expressed also as
		\begin{equation}
			M(x,z)=I+C^{\pm}(\mu_x(\zeta)w(\zeta))=I+\frac{1}{2\pi i}\int_{\mathbb{R}}\frac{\mu_x(\zeta)w(\zeta)}{\zeta-z}d\zeta.
		\end{equation}
	Then differentiating the jump relation $M_+(z)=M_-(z)V(z)$, we obtain
	\begin{equation*}
		\frac{d}{dx}M_++iz[M_+,\sigma_3]=(\frac{d}{dx}M_-+iz[M_-,\sigma_3])V.
	\end{equation*}
		A simple calculation shows that
$$iz[M_{\pm},\sigma_3]=Q+C^{\pm}(i[\mu_xw,\sigma_3]).$$
 Thus $M_{\pm}$ solve the differential equation $\frac{d}{dx}M_{\pm}=iz[\sigma_3,M_{\pm}]+QM_{\pm}$. Set $\mathcal{I}(r)=q$. The following results show that $\mathcal{I}$ maps  $H^{s}(\mathbb{R})\,\cap\, L^{2,1}(\mathbb{R})$ to $H^1(\mathbb{R})\,\cap\,L^{2,s}(\mathbb{R})$.
		For a fixed $c_s$ and $x\leq0$, using the argument in Lemma 3.4 \cite{20} we obtain
		\begin{equation}
			||C^{\pm}w(z)||_{L_z^2(\mathbb{R})}\leq c_s\langle x\rangle^{-s}||r||_{H^s(\mathbb{R})\,\cap\,L^{2,1}(\mathbb{R})},
		\end{equation}
	 which can be induced by the different definition of the operator in (\ref{eq:2.1}) from which we get
		\begin{equation}
			||C_wI||_{L^2(\mathbb{R})}\leq2c_s\langle x\rangle^{-s}||r||_{H^s(\mathbb{R})\,\cap\,L^{2,s}(\mathbb{R})}.
		\end{equation}
		Then we take into account the inequality
		\begin{equation}
			\mu_x-I=(I-C_w)^{-1}C_wI,
		\end{equation}
		and correspondingly
		\begin{equation}
			||\mu_x-I||_{L_z^2}\leq||(I-C_w)^{-1}||_{L_z^2\to L_z^2}||C_wI||_{L_z^2}.
		\end{equation}
		For fixed $c$ and from the Lemma 5.1 \cite{24} we obtain $||(I-C_w)^{-1}||_{L_z^2\to L_z^2}\leq C\langle\rho\rangle^2$ where $\rho:=||r||_{L^\infty(\mathbb{R})}$. We conclude that for $x\leq0$ and for any $\kappa_0$ there is a constant $C$ such that
		\begin{equation}
			||\mu_x-I||_{L_z^2}\leq C\langle x\rangle^{-s}||r||_{H^s(\mathbb{R})\,\cap\,L^{2,s}(\mathbb{R})},
		\end{equation}
		for $\rho\leq\kappa_0$.
		As above, we write
		\begin{equation*}
		M-I	=\int_{\mathbb{R}}((I-C_w)^{-1}(V_+-V_-))dz=\int_1+\int_2+\int_3,
		\end{equation*}
		where
		\begin{equation*}
			\int_1=\int(V_+-V_-),\qquad\qquad\int_2=\int(C_wI)(V_+-V_-),
		\end{equation*}	
		\begin{equation*}
			\int_3=\int(C_w)(I-C_w)^{-1}C_wI)(V_+-V_-)=\int(C_w)(\mu-I)(V_+-V_-).
		\end{equation*}
		We remark that for calculating $q$, the estimate of $\int_2$ is not needed because it is diagonal and $\hat{\sigma_3}\int_2=0$. But the estimate is useful for other problems. Clearly, $\int(V_+-V_-)\in H^s(\mathbb{R})\cap L^{2,1}(\mathbb{R})$ by the Fourier transform. Using the triangularity of $V_{\pm}$, the fact that $C^+-C^-=1$ Cauchy's theorem, using Lemma 3.4 \cite{20} we obtain for some $c>0$
		\begin{align*}
			|\int_2|&=|\int[(C^+(I-V_-))(V_+-I)+(C^-(V_+-I)(I-V_-))]\\
			&=|\int[C^+(I-V_-))(C^-(I-V_+))+(C^-(V_+-I))(C^+(I-V_-))]\\
			&\leq c(1+x^2)^{-1},
			\end{align*}
	and
	\begin{equation*}
		||\mu-I||_{L^2}=||(I-C_w)^{-1}C_wI||_{L^2}\leq c(1+x^2)^{-1/2},
	\end{equation*}
		we have
		\begin{align*}
			|\int_3|&=|\int[(C^+(\mu-I)(I-V_-)(V_+-I))]+(C^-(\mu-I)(V_+-I)(I-V_-))\\
		            &\leq c(1+x^2)^{-1}.
		\end{align*}
According to Theorem 3.6 \cite{20}, we  have $q \in  H^1(\mathbb{R})\,\cap\,L^{2,s}(\mathbb{R}),$
moreover,
 $$||q||_{H^1(\mathbb{R})\,\cap\,L^{2,s}(\mathbb{R})}\leq C||r||_{H^s(\mathbb{R})\,\cap\, L^{2,1}(\mathbb{R})}.$$
		\\
		For the case  when  $z\in (z_1,z_2)$,  we consider   the second  decomposition
		\begin{equation}
			V=   \begin{pmatrix}
				1&0\\
				\frac{r(z)e^{2it\theta}}{1+|r(z)|^2}&1
			\end{pmatrix}
			\begin{pmatrix}
				1+|r(z)|^2&0\\
				0&\frac{1}{1+|r(z)|^2}
			\end{pmatrix}
			\begin{pmatrix}
				1&\frac{r^*(z)e^{-2it\theta}}{1+|r(z)|^2}\\
				0&1
			\end{pmatrix}. \nonumber
		\end{equation}
	Further the RH problem \ref{RH21} can be  changed into  the RH problem \ref{RH31}	
		by using transformation $M^{(1)}=M\delta^{-\sigma_3}$.  Correspondingly we acquire estimates
$$||\tilde{q}||_{H^1(\mathbb{R})\,\cap\,L^{2,s}(\mathbb{R})}\leq C||\tilde{r}||_{H^s(\mathbb{R})\cap L^{2,1}(\mathbb{R})}\\\leq c||r||_{H^s(\mathbb{R})\cap L^{2,1}(\mathbb{R})}, $$
 for the function $\tilde{q}=  \mathcal{I}(\tilde{r}) $  with  $\tilde{r}:=r\delta_+\delta_-$ and for fixed $c$ when $\rho\leq\kappa_0$ by proceeding as above. Finally, $\tilde{q}=q$ and more details see \cite{20}.
	\end{proof}
	
	% \noindent\textbf{\expandafter{Lemma 2.3} }

	% Now consider the representation of the solutions of the Cauchy problem (\ref{eq:1.1}) in terms of the inverse scattering transform. For fixed $s\in(\frac{1}{2},1]$ we suppose that $q_0\in H^1(\mathbb{R}\cap L^{2,s}(\mathbb{R})$ and according to the Theorem 1.1 we obtain the solution remains in $q(x)\in H^1(\mathbb{R}\cap L^{2,s}(\mathbb{R})$ for all $t\in\mathbb{R}^+$ by standard argument [*]. Then, we consider the time evolution of the scattering data is well defined,

	\section{\noindent\textbf\expandafter{Dispersion for pure radiation solutions}	}\label{sec3}

	  In this section, we consider the elements of $\mathcal{G}$ such that $\mathcal{Z}=\varnothing$ generate pure radiation solutions of the Hirota equation.

\subsection{A regular RH problem }

From   the function (\ref{eq:2.4}),  we get two stationary points
	\begin{align}
		z_1=\frac{-\alpha-\sqrt{\alpha^2-3\beta x/t}}{6\beta},\\
		z_2=\frac{-\alpha+\sqrt{\alpha^2-3\beta x/t}}{6\beta}.
	\end{align}
	The signature table of Re$(it\theta(z))$ is given in Figure \ref{F1}.

	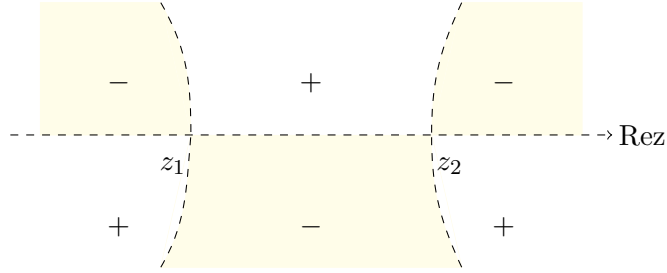
\begin{figure}
		\begin{center}
			\begin{tikzpicture} [scale=0.8]
				\draw [yellow!10, fill=yellow!10] (2,0)--(4.5,0)--(4.5,2.2)--(2.5,2.2);
				\draw [yellow!10, fill=yellow!10] (-2,0)--(-4.5,0)--(-4.5,2.2)--(-2.5,2.2);
				\draw [yellow!10, fill=yellow!10] (-2,0)--(-2.5,-2.2)--(2.5,-2.2)--(2,0);
				\draw [dashed,-> ](-5,0)--(5,0);
				\draw [dashed, fill=yellow!10] (2.5,2.2) to [out=-115,in=90] (2,0);
				\draw [dashed, fill=white!10] (2.5,-2.2) to [out=115,in=-90] (2,0);
				\draw [dashed, fill=yellow!10] (-2.5,2.2) to [out=-60,in=90] (-2,0);
				\draw [dashed, fill=white!10] (-2.5,-2.2) to [out=60,in=-95] (-2,0);
				% \draw [red]( 0,-2)--(4,2);
				% \draw [red]( 0,-2)--(-4, 2);
				%\draw [blue]( 0, 2)--(4,-2);
				% \draw [ blue]( 0, 2)--(-4,-2);
				%\draw [dashed  ](0,-2)--(0,2);
				%\draw[-latex,red] ( 0,-2)--(1,-1);
				% \draw[-latex,blue] (-2,0)--(-1,1);
				%\draw[-latex,red] (4,2)--(3,1);
				%\draw[-latex,blue] ( 4,-2)--(3,-1);
				% \draw[-latex,red] (-2,0)--(-3,1);
				% \draw[-latex,blue] (-2,0)--(-3,-1);
				%\draw[-latex,blue] ( 0,2)--( 1,1);
				% \draw[-latex,red] (-2,0)--(-1,-1);
				\node    at (5.5,0)  { Rez};
				\node  [below]  at (3.2,1.2) {$-$};
				\node  [below]  at (-3.2,1.2) {$-$};
				\node  [below]  at (-3.2,-1.2) {$+$};
				\node  [below]  at (3.2,-1.2) {$+$};
				\node  [below]  at (0,-1.2) {$-$};
				\node  [below]  at (0,1.2) {$+$};
				
				% \node  [below]  at (3.2,-0.2) {$\Omega_{26}$};
				% \node  [below]  at (0,2.6) {$\Omega_2$};
				\node  [below]  at (2.3,-0.2) {$z_2$};
				%  \node  [below,]  at (1,0.7) {$\Omega_{23} $};
				% \node  [below]  at (1,-0.1) {$\Omega_{24} $};
				% \node  [below]  at (-3.2,0.8) {$\Omega_{11}$};
				% \node  [below]  at (-3.2,-0.2) {$\Omega_{16}$};
				% \node  [below]  at (0,-2.2) {$\Omega_5$};
				\node  [below]  at (-2.3,-0.2) {$ z_1$};
				%   \node  [below]  at (-1,0.7) {$\Omega_{13}$};
				% \node  [below]  at (-1,-0.1) {$\Omega_{24}$};
				%\node  [below]  at (4,2.6) {$\Sigma_{21}$};
				%\node  [below]  at (-4,2.6) {$\Sigma_{11}$};
				% \node  [below]  at (4,-2.2) {$\Sigma_{22}$};
				%\node  [below]  at (-4,-2.2) {$\Sigma_{12}$};
				%\node  [below]  at (1.3,1.5) {$\Sigma_{23}$};
				% \node  [below]  at (-1.3,1.5) {$\Sigma_{13}$};
				%\node  [below]  at ( 1.2,-1.1) {$\Sigma_{24}$};
				%  \node  [below]  at (-1.2,-1.1) {$\Sigma_{14}$};
				%  \node[right] at (0,0.8){$l_1$};
				%   \node[right] at (0,-0.8){$l_2$};
				%\node  [below]  at (2.5,0) {$z_0$};
				%\node  [below]  at (-2.5,0) {$-z_0$};
			\end{tikzpicture}
		\end{center}
		\caption{The    function $e^{2 i t\theta}$ decay  in yellow domains and increase in white domains.   }
		\label{F1}
	\end{figure}
	
	First we consider the scalar factorization problem
	\begin{equation}\label{eq:3.5}
		\begin{cases}
			\delta_+(z)=\delta_-(z)(1+|r(z)|^2),\quad z\in(z_1,z_2),\\
			\delta_+(z)=\delta_-(z),\quad\quad\qquad\qquad\,\, \mathbb{R}\setminus(z_1,z_2),\\
			\delta(z)\to1,\qquad\qquad\qquad\,\,\,\,\qquad z\to\infty,
		\end{cases}
	\end{equation}		
	 which admits a solution  by the Plemelj formula
	\begin{align}
		\delta(z)=\exp\left[i\int_{z_1}^{z_2}\frac{\nu(s)}{s-z}ds\right],
	\end{align}
	where
	\begin{align}\label{eq:3.6}
		&\nu(s)=-\frac{1}{2\pi}\log(1+|r(s)|^2).
	\end{align}
	Define $z_0=\frac{|z_2-z_1|}{2}$ then function $\delta(z)$ has an expansion
	\begin{align}\nonumber
		\delta(z)&=\exp\big(i\int_{z_1}^{z_2}\frac{\nu(s)-\chi(s)\nu(z_2)(s-z_2+1)}{s-z}ds+i\nu(z_2)\int_{z_0}^{z_2}\frac{s-z_2+1}{s-z}ds\\\nonumber
		&=\exp\{i\beta(z,z_2)+i\nu(z_2)+i\nu(z_2)[(z-z_2)\log(z-z_2)\\\nonumber
		&-(z-z_2+1)\log(z-z_2+1)+i\nu(z_2)\log(z-z_2)\}\\\nonumber
		&=e^{i\nu(z_2)+i\beta(z,z_2)}(z-z_2)^{i\nu(z_2)}e^{i\nu(z_2)[(z-z_2)\log(z-z_2)-(z-z_2+1)\log(z-z_2+1)]},
	\end{align}
	where $\chi(s)$ is characteristic function of $(z_0,z_2)$
	$$\beta(z,z_2)=\int_{z_1}^{z_2}\frac{\nu(s)-\nu(z_2)\chi(s)(s-z_2+1)}{s-z}ds.$$
	Let $\rho:=||r||_{L^\infty}(\mathbb{R})$, the function $\delta(z)$ satisfying the properties following:
	\begin{itemize}
		\item  when $z\notin[z_1,z_2]$, we obtain $\delta(z)\delta^*(z^*)=1$ and $\langle\rho\rangle^{-1}\leq|\delta(z)|\leq\langle\rho\rangle$;
		\item For $\mp\mbox{Im}z>0$ we have $|\delta^{\pm}(z)|\leq1$.
	\end{itemize}
	
Next  we just consider the stationary point
  $z=z_2$, and  the point $z=z_1$ is similar.
	\begin{lemma}\label{le31}Define $L_\phi=\tilde{z}+e^{-i\phi}\mathbb{R}=\{z=\tilde{z}+e^{-i\phi}q:q\in\mathbb{R}\}$. For $s\in(\frac{1}{2},1]$ then there is a fixed $C(\rho,s)$ such that that for any $\tilde{z}\in\mathbb{R}$ and any $\phi\in(0,\pi)$
		\begin{align}\label{eq:3.12}
			&||\beta(e^{-i\phi}\cdot,\tilde{z})||_{H^s(\mathbb{R})}\leq C(\rho,s)||r||_{H^s(\mathbb{R})},\\\label{eq:3.12b}
			&|\beta(z,\tilde{z})-\beta(\tilde{z},\tilde{z})|\leq C(\rho,s)||r||_{H^s(\mathbb{R})\,\cap\, L^{2,1}(\mathbb{R})}|z-\tilde{z}|^{s-\frac{1}{2}},\,\,z\,\in L_{\phi}.
		\end{align}
	\end{lemma}
	\begin{proof} When $s=1$ there have
$$||C_{\mathbb{R}}f||_{H^\tau(L_\phi)}\leq C_\tau||f||_{H^\tau(\mathbb{R})}, \ \tau=0,1$$
 which are proved in Lemma 23.3 \cite{26}. By interpolation calculation we get  the case $\tau=s$ where $s\in(0,1)$. And the (\ref{eq:3.12b}) can be obtained from (\ref{eq:3.12}) of the following elementary estimate when $s\in(1/2,1]$
		\begin{equation}\label{eq:3.14}
			|f(x)-f(y)|\leq C_s||f||_{H^s(\mathbb{R})}|x-y|^{s-\frac{1}{2}},
		\end{equation}
		for all $x,y\in\mathbb{R}$ and $f\in H^s(\mathbb{R})$ for a fixed $C_s$.
Noting that
$$f(x+h)-f(x)=\frac{1}{\sqrt{2\pi}}\int e^{ix\xi}(e^{ih\xi}-1)\hat{f}(\xi)d\xi,  $$
 Then for any $\kappa>0$ we have for a fixed $C_s$
		\begin{align}\nonumber
			|f(x+h)-f(x)|&\leq\frac{||f||_{H^s}}{\sqrt{2\pi}}\Big[\big(|h|\int_{|\xi|\leq\kappa}|\xi|^{2-2s}d\xi\big)^{\frac{1}{2}}+\big(\int_{|\xi|\ge\kappa}|\xi|^{-2s}d\xi\big)^{\frac{1}{2}}\Big]\\\nonumber
			&\leq C_s(|h|\kappa^{\frac{3-2s}{2}}+\kappa^{\frac{1-2s}{2}})||f||_{H^s}.
			\end{align}
	By Plancherel formula we have
	\begin{align}\nonumber
			|f(x+h)-f(x)|&\leq C_s\frac{||f||_{L^{2,1}}}{\sqrt{2\pi}}\Big[\big(|h|\int_{|\xi|\leq\kappa}|\xi|^{1-2s}d\xi\big)^{\frac{1}{2}}+\big(\int_{|\xi|\ge\kappa}|\xi|^{-2s}d\xi\big)^{\frac{1}{2}}\Big]\\
			&\leq C_s\kappa^{\frac{1-2s}{2}}||f||_{L^{2,1}},
		\end{align}
		which equals $2C_s|h|^{s-\frac{1}{2}}||f||_{H^s}$ for $\kappa=|h|^{-1}$.\end{proof}

	We define a new unknown function
	 \begin{align}
M^{(1)}(z)=M(z)\delta^{-\sigma_3}(z),\label{trns1}
\end{align}
which satisfies  the following RH problem.
	\begin{problem}\label{RH31}
		Find an analytic function $M^{(1)}(z):\mathbb{C}\setminus\mathbb{R}\to SL_2(\mathbb{C})$ with the following properties
		
		\begin{enumerate}
			\item $ M^{(1)}(z)=I+\mathcal{O}(z^{-1})$ as $z\to\infty$.
			\item $M^{(1)}(z)$ takes continuous boundary values $M^{(1)}_{\pm}(z)$  and  satisfy the jump relation
$$M_+^{(1)}(z)=M_-^{(1)}(z)V^{(1)}(z), $$
 where
			\begin{align}\nonumber
				& V^{(1)}(z)=\delta_-^{\sigma_3}V(z)\delta_+^{-\sigma_3}\\\nonumber
				\,
				& =\begin{cases}\label{eq:3.7}
					\begin{pmatrix}
						1&r^*(z)\delta^2(z)e^{-2it\theta(z)}\\
						0&1
					\end{pmatrix}\begin{pmatrix}
						1&0\\
						r(z)\delta^{-2}(z)e^{2it\theta}&1\end{pmatrix},\quad z\in(-\infty,z_1)\cup(z_2,\infty),\\
					\\
					\begin{pmatrix}
						1&0\\
						\frac{r(z)\delta^{-2}e^{2it\theta}}{1+|r(z)|^2}&1
					\end{pmatrix}\begin{pmatrix}
						1&\frac{r^*(z)\delta^{-2}e^{-2it\theta}}{1+|r(z)|^2}\\
						0&1
					\end{pmatrix},\,\,\qquad\qquad\quad z\in(z_1,z_2).
				\end{cases}
			\end{align}
		\end{enumerate}
	\end{problem}
	
	\subsection{A mixed  RH  Problem and its decomposition}
	\quad\, In this section we construct a mixed  RH Problem
	by follow closely the argument of  \cite{25,20,27}.
	
	Fix a smooth cut-off function of compact support, with $\chi(x)\geq0$ for any $x$ and $\int\chi dx=1$. Let $\chi_\epsilon(x)=\epsilon^{-1}\chi(\epsilon^{-1}x)$ for $\epsilon\neq0$. We define $\mathbf{r}(z)$ as follows
	\begin{equation}
		\mathbf{r}(z)=\begin{cases}
			r(\mbox{Re}z),\,\quad\quad\,\,\,\,\,\,\,\,\mbox{for}\,\,\mbox{Im}z=0,\\
			\chi_{\mbox{Imz}}*r(\mbox{Re}z),\,\,\mbox{for}\,\,\mbox{Im}z\ne0.
		\end{cases}
	\end{equation}
	For convenience of expression, for $j=1,2$ we define rays
\begin{align}
&L_j=\{z_j+\mathbb{R}^+e^{i\frac{\pi}{4}}\}\cup\{z_j+\mathbb{R}^+e^{i\frac{5\pi}{4}}\}=\Sigma_{j2}\cup\Sigma_{j3},\\
&\overline{L}_j=\{z_j+\mathbb{R}^+e^{-i\frac{\pi}{4}}\}\cup\{z_j+\mathbb{R}^+e^{i\frac{3\pi}{4}}\}=\Sigma_{j1}\cup\Sigma_{j4},\\
&I_{12}=(z_1,0),\quad I_{22}=(0,z_2),\,\,\,\,I_{11}=(-\infty,z_1),\,\,\,\,I_{21}=(z_2,\infty).\nonumber
\end{align}
These rays divide the complex plane $\mathbb{C}$ into ten domains $\Omega_{ij}, i=1,2; j=1,3,4,6; \Omega_2, \Omega_5$.
See  Figure \ref{fsptrum}.
	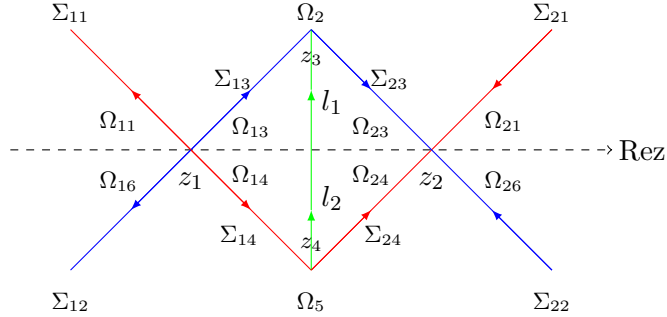
\begin{figure}
		\begin{center}
			\begin{tikzpicture}[scale=0.8]
				%\draw [yellow] (2,0)--(4.5,0)--(4.5,2.2)--(2.5,2.2);
				%\draw [yellow!10, fill=yellow!10] (-2,0)--(-4.5,0)--(-4.5,2.2)--(-2.5,2.2);
				%\draw [yellow!10, fill=yellow!10] (-2,0)--(-2.5,-2.2)--(2.5,-2.2)--(2,0);
				\draw [dashed,-> ](-5,0)--(5,0);
				%\draw [dashed] (2.5,2.2) to [out=-115,in=90] (2,0);
				%\draw [dashed] (2.5,-2.2) to [out=115,in=-90] (2,0);
				%\draw [dashed] (-2.5,2.2) to [out=-60,in=90] (-2,0);
				%\draw [dashed] (-2.5,-2.2) to [out=60,in=-95] (-2,0);
				\draw [red]( 0,-2)--(4,2);
				\draw [red]( 0,-2)--(-4, 2);
				\draw [blue]( 0, 2)--(4,-2);
				\draw [ blue]( 0, 2)--(-4,-2);
				%\draw [dashed  ]
				\draw[-latex,green] (0,-2)--(0,-1);
				\draw[-latex,green] (0,-1)--(0,1);
				\draw[green](0,1)--(0,2);
				\draw[-latex,red] ( 0,-2)--(1,-1);
				\draw[-latex,blue] (-2,0)--(-1,1);
				\draw[-latex,red] (4,2)--(3,1);
				\draw[-latex,blue] ( 4,-2)--(3,-1);
				\draw[-latex,red] (-2,0)--(-3,1);
				\draw[-latex,blue] (-2,0)--(-3,-1);
				\draw[-latex,blue] ( 0,2)--( 1,1);
				\draw[-latex,red] (-2,0)--(-1,-1);
				\node    at (5.5,0)  { Rez};
				\node  [below]  at (3.2,0.8) {\footnotesize $\Omega_{21}$};
				\node  [below]  at (3.2,-0.2) {\footnotesize $\Omega_{26}$};
				\node  [below]  at (0,2.6) {\footnotesize $\Omega_2$};
				\node  [below]  at (0,1.8) {\footnotesize $z_3$};
				\node  [below]  at (0,-1.3) {\footnotesize $z_4$};
				\node  [below]  at (2,-0.2) {$z_2$};
				\node  [below,]  at (1,0.7) {\footnotesize $\Omega_{23} $};
				\node  [below]  at (1,-0.1) {\footnotesize $\Omega_{24} $};
				\node  [below]  at (-3.2,0.8) {\footnotesize $\Omega_{11}$};
				\node  [below]  at (-3.2,-0.2) {\footnotesize $\Omega_{16}$};
				\node  [below]  at (0,-2.2) {\footnotesize $\Omega_5$};
				\node  [below]  at (-2,-0.2) {$ z_1$};
				\node  [below]  at (-1,0.7) {\footnotesize $\Omega_{13}$};
				\node  [below]  at (-1,-0.1) {\footnotesize $\Omega_{14}$};
				\node  [below]  at (4,2.6) {\footnotesize  $\Sigma_{21}$};
				\node  [below]  at (-4,2.6) {\footnotesize $\Sigma_{11}$};
				\node  [below]  at (4,-2.2) {\footnotesize $\Sigma_{22}$};
				\node  [below]  at (-4,-2.2) {\footnotesize $\Sigma_{12}$};
				\node  [below]  at (1.3,1.5) {\footnotesize $\Sigma_{23}$};
				\node  [below]  at (-1.3,1.5) {\footnotesize $\Sigma_{13}$};
				\node  [below]  at ( 1.2,-1.1) {\footnotesize $\Sigma_{24}$};
				\node  [below]  at (-1.2,-1.1) {\footnotesize $\Sigma_{14}$};
				\node[right] at (0,0.8){$l_1$};
				\node[right] at (0,-0.8){$l_2$};
				%\node  [below]  at (2.5,0) {$z_0$};
				%\node  [below]  at (-2.5,0) {$-z_0$};
			\end{tikzpicture}
		\end{center}
		\caption{The continuous extension  domains $\Omega_{nj}, \ n=1, 2; j=1,3,4,6$ and  contour $\Sigma^{(2)}$  for mixed RH problem $M^{(2)}(z)$  }
		\label{fsptrum}
	\end{figure}
	
We define functions $R_{ij} :\overline{\Omega}_{ij}\to\mathbb{C},(i=1,2; j=1,3,4,6)$ and a constant $c$ with boundary values satisfying
	\begin{align}\label{eq:R1}
		&R_{j1}(z)=\begin{cases}
			r(z),\qquad\qquad\qquad\qquad\quad\qquad\quad\qquad\qquad\,\,\, z\in I_{j1},\\
		    \hat{r}_0(z-z_j)^{-2i\nu(z_j)}\delta^2,\quad\quad \quad\,\,\,\,\,\,\qquad\qquad\qquad z\in\Sigma_{j1},
		\end{cases}\\
		&R_{j3}(z)=\begin{cases}\label{eq:R2}
			\frac{r^*(z)}{1+|r(z)|^2},\qquad\qquad\qquad\,\,\qquad\qquad\qquad\, \,\quad\,\,\,\,\,\,\,z\in I_{j2},\\
			\frac{\hat{r}_0^*(z)}{1+|r(z_j)|^2}(z-z_j)^{2i\nu(z_j)}\delta^{-2},\,\,\,\,\,\,\,\,\quad\quad\qquad\qquad\,\, z\in\Sigma_{j3},
		\end{cases}\\
		&R_{j4}(z)=\begin{cases}\label{eq:R3}
			\frac{r(z)}{1+|r(z)|^2},\qquad\qquad\qquad\,\qquad\qquad\qquad\, \,\quad\,\,\,\,\,\,\,\,\,z\in I_{j4},\\
			\frac{\hat{r}_0(z)}{1+|r(z_j)|^2}(z-z_j)^{-2i\nu(z_j)}\delta^{2},\,\,\,\,\,\,\,\,\,\qquad\qquad\qquad z\in\Sigma_{j4},
		\end{cases}\\
		&R_{j6}(z)=\begin{cases}\label{eq:R4}
			r^*(z),\qquad\qquad\qquad\,\, \,\,\qquad\quad\qquad\qquad\qquad\,\,\,\, z\in I_{j6},\\
		    \hat{r}_0^*(z-z_j)^{2i\nu(z_j)}\delta^{-2},\quad\qquad\,\,\,\,\,\,\,\, \qquad\qquad\qquad z\in\Sigma_{j6},
		\end{cases}
	\end{align}
	where $\hat{r}_0=r(z_j)e^{-2i\nu(z_j)-2\beta(z_j,z_j)}$ .
	\begin{proposition}\label{pr31} 	Fix $\lambda_0>0$ and assume $||r||_{H^s(\mathbb{R})\,\cap\, L^{2,1}(\mathbb{R})}<\lambda$ a preassigned $s\in(1/2,1]$. Then for $j=1,2;j=1,2,3,4$ there exist functions $R_{ji}: \overline{\Omega_{ji}}\to\mathbb{C}$
		such that forall $z\in\Omega_{ji}+z_j$ and $\psi(x)=-\chi(x)-x\chi'(x)$, we have for a fixed $c$
		\begin{align}\nonumber
			|\overline\partial R_{ji}(z)|&\leq c||r||_{H^s(\mathbb{R})\,\cap L^{2,1}(\mathbb{R})}|z-z_j|^{s-\frac{3}{2}}\\\label{eq:3.19}
			&+c|\partial_{\tiny\rm \mbox{Re}z}\mathbf{r}(z)|+c|(\mbox{Im z})^{-1}\psi_{\tiny\rm \mbox{Im}z}*r(\mbox{Re z})|.
		\end{align}
	\end{proposition}
	\begin{proof} We consider the stationary point at $z_2$, the case at  $z_1$ is similar. For $\zeta\in\Omega_{13}$ we use $0<\arg(\zeta-z_1)<\pi/4$ and set $G(\zeta):=g(\arg(\zeta-z_1))$ where the function $g: [0,\pi/4]\to[0,1]$ is a smooth function such that $g(\psi)=1$ for all $\psi\in[0,\pi/6]$ and $g(\pi/4)=0$. the function $G$ is continuous in $\Omega_{13}\setminus\{z_1\}$ and fulfills
		\begin{equation}
		G(\zeta)=\begin{cases}
				1,\quad\mbox{for}\,\zeta\in I_{12},\\
				0,\quad\mbox{for}\,\zeta\in \Sigma_{13},\\
			\end{cases}
		\end{equation}
		Moreover, we can find a constant $c$ such that
		$$|\overline\partial G(\zeta)\leq c|\zeta-z_1|^{-1}$$
	
		Consider the function of $R_{2i}(z)$ which can be defined explicitly by setting for $z-z_2=u+iv$. For $i=1,3$ we obtain
		\begin{align}\nonumber
			&R_{21}(z)=G(\zeta)\mathbf{r}(z)+(1-G(\zeta)))f_{21}(u+iv),\\\nonumber
			&R_{23}(z)=\cos(2(\arg(z-z_0)-\pi))\frac{\mathbf{r}^*(z)}{1+|\mathbf{r}(z)|^2}\\\label{eq:3.13}
			&\quad\qquad\,\,+(1-\cos(2\arg(z-z_0)-\pi))f_{23}(u+iv).
		\end{align}
		For $i=4,6$ which can be defined similarly. Now we consider the boundary values and consider case $i=1$ only. We have
		$$\overline\partial R_{21}=(\mathbf{r}-f_{21})\overline\partial b+\frac{b}{2}(\chi_{\mbox{Im}z}*r'(\mbox{Re}z)+i(\mbox{Im}z)^{-1}\psi_{\mbox{Im} z}*r(\mbox{Re}z)),$$
		with $\psi(x)=-\chi(x)-x\chi'(x)$ and $f_{21}	\frac{\hat{r}_0^*(z)}{1+|r(z_1)|^2}(z-z_1)^{2i\nu(z_j)}\delta^{-2}(1-\Xi(\zeta))$.	Notice that $\hat\psi(0)=0$. Then we have the bound
		\begin{align}\nonumber
			|\overline\partial R_{21}|&\leq|\chi_{\mbox{Im}z}*r'(\mbox{Re}z)|+|(\mbox{Im}z)^{-1}\psi_{\mbox{Im}z}*r(\mbox{Re}z)|\\
			&+\frac{c}{|z-z_2|}|(r(z)-r(z_2)|+|f_{21}(z)-r(z_2)|).
		\end{align}
		By (\ref{eq:3.14}) we acquire that
$$|r(z)-r(z_2)|\leq C|z-z_2|^{s-\frac{1}{2}}||r||_{H^s(\mathbb{R})\,\cap\,L^{2.s}(\mathbb{R})}$$
 to get the desired estimate for $|\overline\partial R_{21}|$ we need to bound the last line. Then, we have
		\begin{align}\nonumber
			&f_{21}(z)-r(z_2)=r(z_2)\times[\exp(2i\mu(z_2)((z-z_2)\log(z-z_2)\\
			&-(z-z_2+1)\log(z-z_2+1))+2(\beta(z,z_2)-\beta(z_2,z_2)))-1].
		\end{align}
		According to Lemma \ref{le31} we have
$$\beta(z,z_2)-\beta(z_2,z_2)=C(\rho,s)||r||_{H^s(\mathbb{R})\cap L^{2,1}(\mathbb{R})}|z-z_2|^{s-\frac{1}{2}},$$
	  and notice that both $(z-z_2)\log(z-z_2)$ and $(z-z_2+1)\log(z-z_2+1)$ are $\mathcal{O}(|z-z_2|^{s-\frac{1}{2}})$ when $z\to z_2$. We get the desired estimate for $|\overline\partial R_{21}|$.
	\end{proof}
	
	Based on  Proposition \ref{pr31}, we  define a new unknown function
	\begin{align}
		&\mathcal{R}^{(2)}=\begin{cases}\label{eq:3.16}
			\begin{pmatrix}
				1&0\\
				R_{j1}e^{2it\theta}\delta^{-2}&1
			\end{pmatrix}^{-1},\qquad z\in\Omega_{j1},\\[8pt]
			\begin{pmatrix}
				1&R_{j3}e^{-2it\theta}\delta^{2}\\
				0&1
			\end{pmatrix}^{-1},\quad\,\,\, z\in\Omega_{j3},\\[8pt]
			\begin{pmatrix}
				1&0\\
				R_{j4}e^{2it\theta}\delta^{-2}&1
			\end{pmatrix},\qquad\quad\, z\in\Omega_{j4},\\[8pt]
			\begin{pmatrix}
				1&R_{j6}e^{-2it\theta}\delta^2\\
				0&1
			\end{pmatrix},\qquad\quad\,\, z\in\Omega_{j6},
			\\[5pt]
			\begin{pmatrix}
				1&0\\
				0&1
			\end{pmatrix},\qquad\qquad\qquad\quad\,\,\,\,\, z\in\Omega_{2}\cup\Omega_{5}.
		\end{cases}
	\end{align}
We introduce  jump contour
 $$\Sigma^{(2)}=L_1\cup\overline{L}_1\cup L_2\cup\overline{L}_2\cup l_1\cup l_2,$$
 then  it can be shown that
\begin{align}
 M^{(2)}(z)=M^{(1)}(z)\mathcal{R}^{(2)}\label{eq:3.15}
\end{align}
 satisfies the following RH problem.

	%\noindent\textbf{\expandafter{Riemann-Hilbert Problem 3.2} }\label{RH32}

	\begin{problem}\label{RH32}
		Find  a continuous function $M^{(2)}(z):\mathbb{C}\setminus\Sigma^{(2)}\to SL_2(\mathbb{C})$ with the following properties
		\begin{enumerate}
			\item $ M^{(2)}(z)=I+\mathcal{O}(z^{-1})$ as $z\to\infty$.
			\item For $z\in\Sigma^{(2)}$, the boundary values satisfy the jump relation
$$M^{(2)}_+(z)=M^{(2)}_-(z)V^{(2)}(z),$$
 where
			\begin{align}\label{eq:3.26}
				& V^{(2)}(z)=\begin{cases}
					\begin{pmatrix}
						1&0\\
						R_{j1}(z)\delta^{-2}(z)e^{2it\theta}&1\end{pmatrix},\quad\,\, z\in\Sigma_{j1},\\
					\\
					\begin{pmatrix}
						1&-R_{j3}(z)\delta^{2}e^{-2it\theta}\\
						0&1
					\end{pmatrix},\quad\quad z\in\Sigma_{j3},
					\\ \\
					\begin{pmatrix}
						1&0\\
						R_{j4}(z)\delta^{-2}e^{2it\theta}&1
					\end{pmatrix},\quad\,\quad\,\,\,\,\, z\in\Sigma_{j4},\\
					\\
					\begin{pmatrix}
						1&-R_{j6}(z)\delta^2e^{-2it\theta}\\
						0&1
					\end{pmatrix},\quad\,\,\quad z\in\Sigma_{j6},\\
					\\	\begin{pmatrix}
						1&(R_{23}-R_{13})e^{-2it\theta}\delta^2\\
						0&1
					\end{pmatrix},\qquad z\in l_1,\\
					\\
					\begin{pmatrix}
						1&0\\
						(R_{14}-R_{24})e^{2it\theta}\delta^{-2}&1
					\end{pmatrix},\qquad z\in l_2.\\
				\end{cases}
			\end{align}
			\item For $\mathbb{C}\setminus\Sigma^{(2)}$ we have
$$\overline\partial M^{(2)}(z)=M^{(2)}(z)\overline\partial\mathcal{R}^{(2)}(z),$$
where
			\begin{equation}\label{eq:3.11}
				\overline\partial\mathcal{R}^{(2)}=\begin{cases}
					\begin{pmatrix}
						1&0\\
						-\overline\partial R_{j1}(z)e^{2it\theta}\delta^{-2}&1
					\end{pmatrix},\qquad\,\, z\in\Omega_{j1},\\
					\\
					\begin{pmatrix}
						1&-\overline\partial R_{j3}(z)e^{-2it\theta}\delta^{2}\\
						0&1
					\end{pmatrix},\quad\,\,\,\,\,\,\,\, z\in\Omega_{j3},\\
					\\
					\begin{pmatrix}
						1&0\\
						\overline\partial R_{j4}(z)e^{2it\theta}\delta^{-2}&1
					\end{pmatrix},\qquad\quad\, z\in\Omega_{j4},\\
					\\
					\begin{pmatrix}
						1&\overline\partial R_{j6}(z)e^{-2it\theta}\delta^2\\
						0&1
					\end{pmatrix},\qquad\quad\,\, z\in\Omega_{j6},
					\\
					\\
					0,\qquad\,\,\,\qquad\qquad\qquad\quad\qquad\quad z\in\Omega_{j2}\cup\Omega_{j5}.
				\end{cases}
			\end{equation}
		\end{enumerate}
	\end{problem}
	
Next we   decompose $ M^{(2)}$  in the form
\begin{align}
M^{(2)}(z)=M^{(3)}(z) M^{rhp} (z),\label{trans2}
\end{align}
where $M^{rhp} (z)$ as a pure RH problem resulting
from  $M^{(2)}(z) $ by setting $\overline\partial\mathcal{R}^{(2)}\equiv0, \ z\in \mathbb{C}\setminus\Sigma^{(2)}$,   which
satisfies the following RH problem

 \begin{problem}\label{RH320}
		Find  a analytical  function $M^{rhp}(z):\mathbb{C}\setminus\Sigma^{(2)}\to SL_2(\mathbb{C})$ with the following properties
		\begin{enumerate}
			\item $ M^{rhp}(z)=I+\mathcal{O}(z^{-1})$ as $z\to\infty$.
			\item For $z\in\Sigma^{(2)}$, the boundary values satisfy the jump relation
$$M^{rhp}_{+}(z)=MM^{rhp}_{-}(z)V^{(2)}(z), $$
where
 $V^{(2)}(z)$ is given by (\ref{eq:3.26}).
		\end{enumerate}
	\end{problem}
While the  $ M^{(3)}(z)$ defined by (\ref{trans2})   satisfies   the following  RH Problem
	\begin{problem}\label{RH33}
		Find a continuous matrix-valued function $M^{(3)}(z):\mathbb{C}\to SL_2(\mathbb{C})$ with the following properties:
		\begin{enumerate}
			\item $M^{(3)}(z)=I+\mathcal{O}(z^{-1})$ as $z\to\infty$.
			\item For $z\in\mathbb{C}$, we have
			\begin{equation}\nonumber
				\overline\partial M^{(3)}(z)=M^{(3)}(z)W^{(3)}(z),
			\end{equation}
			where $W^{(3)}(z):=M^{rhp}(z) \overline\partial\mathcal{R}^{(2)}(z)M^{rhp}(z)^{-1}$ and $\overline\partial\mathcal{R}^{(2)}(z)$ is defined by (\ref{eq:3.11}).
		\end{enumerate}
	\end{problem}

	\subsection{A solvable   model near $z_1$ and $z_2$}

We introduce  jump contour
 $$\Sigma^{(3)}=L_1\cup\overline{L}_1\cup L_2\cup\overline{L}_2, $$
and consider the following RH problem

 \begin{problem}\label{RH320}
		Find  a analytical  function $M^{loc}(z):\mathbb{C}\setminus\Sigma^{(3)}\to SL_2(\mathbb{C})$ with the following properties
		\begin{enumerate}
			\item $ M^{loc}(z)=I+\mathcal{O}(z^{-1})$ as $z\to\infty$.
			\item For $z\in\Sigma^{(3)}$, the boundary values satisfy the jump relation
$$M^{loc}_{+}(z)=M^{loc}_{-}(z)V^{(2)}(z),$$
 where
 $V^{(2)}(z)$ is given by (\ref{eq:3.26}).
		\end{enumerate}
	\end{problem}
This RH problem is solvable and its solution is given by
 \begin{align*}
		M^{loc}(z)=    M^{pc}(\zeta_1)+  M^{pc}(\zeta_2),
		\end{align*}
where $ M^{pc}(\zeta_j)$ is the solution of the parabolic cylinder model  in  \ref{AA},  and  $\zeta_j=\sqrt{(-1)8t(\alpha+6\beta z_j)}(z-z_j), \ j=1, 2$.

 Direct calculation shows that  there is constant $c>0$ such that
	$$\|V_l-I\|_{L^\infty (l_1\cup l_2)} \lesssim e^{-ct},\quad t\to\infty.$$
Further by using small RH problem, it can be shown that

\begin{proposition}\label{prop34}
		\begin{align*}
			M^{rhp} ( z)= 	M^{loc}  ( z) (I+ O(e^{-ct})).
		\end{align*}

	\end{proposition}

	\subsection{The $\overline\partial$ argument} \label{dbar}
	
	It is well understood that the solution to the  $\overline\partial$ problem \ref{RH33}
 can be given with the following  Cauchy integral
	\begin{equation}
		M^{(3)}(z)=I+\frac{1}{2\pi i}\int_{\mathbb{C}}\frac{M^{(3)}(z)W^{(3)}(z)}{\varsigma-z}d\varsigma. \label{pieee}
	\end{equation}
  Let $||r||_{H^s(\mathbb{R})\,\cap\, L^{2,s}(\mathbb{R})}\leq\lambda_0$,
define  the following operator
		\begin{equation}
			JH(z):=\frac{1}{\pi}\int_{\mathbb{C}}\frac{H(\varsigma)W(\varsigma)}{\varsigma-z}dA(\varsigma),
		\end{equation}
then it can be shown that
\begin{lemma}   The operator $J:L^\infty(\mathbb{C})\to L^\infty(\mathbb{C})\cap C^0(\mathbb{C})$, and there  exist   a $C=C(\lambda_0)$ such that
		\begin{equation}
			||J||_{L^\infty(\mathbb{C})\to L^\infty(\mathbb{C})}\leq Ct^{\frac{1-2s}{4}},\, \,\,t>0.
		\end{equation}
 \end{lemma}

	\begin{proof}
For convenience of expression the proof below only consider the case when $z=z_2$ and from the defination of the operator of $J$ we suppose $H\in L^\infty(\Omega_{21})$ then
		\begin{equation}
			|JH(z)|\leq||H||_{L^\infty(\Omega_{21}})||\delta^{-2}||_{L^\infty(\Omega_{21})}\int_{\Omega_{21}}\frac{|\overline\partial R_{21}(\varsigma)e^{2it\theta}|}{|\varsigma-z|}dA(\varsigma),
		\end{equation}
		and from the properties of $\delta(z)$ we acquire that $||\delta^{-2}||_{L^\infty(\Omega_{21})}\leq1$.
		For $j=1,2,3$ we have the bound of $I_j$ from (\ref{eq:3.19}) following
		\begin{align}\nonumber
			&I_j=\int_{\Omega_{21}}\frac{|X_j(\zeta)e^{2it\theta}|}{|\varsigma-z|}dA(\varsigma),\qquad X_1(z):=\partial_{\mbox{Re} z}\mathbf{r}(z),\\\label{eq:3.45}
			&X_2(z):=||r||_{H^s(\mathbb{R})}|z-z_2|^{s-\frac{3}{2}},\quad X_3(z):=(\mbox{Im}z)^{-1}\psi_{\mbox{Im}z}*r(\mbox{Re}z).
		\end{align}
		Following the prove in Section 2.4  \cite{27} and recall the expression of $z_2$ we set $\varsigma-z_2=u+iv$ and $z-z_2=\alpha+i\beta$, the region $\Omega_{21}$ corresponds to $u\ge v\ge 0$ we have
		\begin{align}\nonumber
			I_1&=\int_{\Omega_{21}}\frac{|\partial_u\mathbf{r}(\varsigma)|e^{-8t(\alpha+6\beta z_2)uv}}{|\varsigma-z|}dudv\lesssim\int_0^\infty dv\int_v^\infty\frac{|\partial_u\mathbf{r}(\varsigma)|e^{-8tuv}e^{-tz_2uv}}{|\varsigma-z|}du\\\label{eq:3.46}
			&\lesssim\int_0^\infty dve^{-8tv^2}e^{-tz_2v^2}||\partial_u\mathbf{r}(u,v)||_{L_u^2(v,\infty)}||((u-\alpha)^2+(v-\beta)^2))^{-1}||_{L_u^2(v,\infty)}.
		\end{align}
		It is easy to check the relationship $||((u-\alpha)^2+(v-\beta)^2)^{-1}||_{L_u^2(v,\infty)}\leq C|v-\beta|^{-\frac{1}{2}}$ is achieved. For fixed $C$ by using the Plancherel we obtain
		\begin{align}\nonumber
			||\partial_u\mathbf{r}(u,v)||_{L_u^2}&=||\partial_u\int_{\mathbb{R}}v^{-1}\chi(v^{-1}(u-t))r(t)dt||_{L_u^2}=||\xi\hat{\chi}(v\xi)\hat{r}(\xi)||_{L^2}\\\label{eq:3.47}
			&\leq v^{s-1}||\xi^{1-s}\hat{\chi}(\xi)||_{L^\infty}||r||_{H^s}\leq Cv^{s-1}||r||_{H^s}.
		\end{align}
	Direct calculation yields
		\begin{align}\nonumber
			I_1&\lesssim||r||_{H^s}\int_{\mathbb{R}}dve^{-8tv^2}e^{-tz_2v^2}|v|^{s-1}|v-\beta|^{-\frac{1}{2}}\\\nonumber
			&\lesssim||r||_{H^s}\int_{\mathbb{R}}dve^{-t(1+z_2)v^2}|v|^{s-1}|v-\beta|^{-\frac{1}{2}}\\\nonumber
			&\lesssim ((1+z_2)t)^{\frac{1-2s}{4}}||r||_{H^s}\int_{\mathbb{R}} dve^{-v^2}(|v|^{s-\frac{3}{2}}+|v-\sqrt{t}\beta|^{s-\frac{3}{2}})\\
			&\lesssim(\int_{\mathbb{R}}e^{-v^2}|v|^{s-\frac{3}{2}}dv)||r||_{H^s}t^{\frac{1-2s}{4}}.
		\end{align}
		For the last inequality we used the fact that for any $c\in\mathbb{R}$
		\begin{align}\nonumber
			&\int_{\mathbb{R}}e^{-v^2}|v-c|^{s-\frac{3}{2}}dv\leq\int_{|v|\leq|v-c|}e^{-v^2}|v|^{s-\frac{3}{2}}dv\\\label{eq:3.49}
			&+\int_{|v|\ge|v-c|}e^{-(v-c)^2}|v-c|^{s-\frac{3}{2}}dv\leq2\int_{\mathbb{R}}e^{-v^2}|v|^{s-\frac{3}{2}}dv.
		\end{align}
		The estimate for $I_3$ is similar after replacing (\ref{eq:3.47}) with
		\begin{align}\nonumber
			||v^{-2}\int&\psi(v^{-1}(u-t))r(t)dt||_{L_u^2}=||v^{-1}\xi^{-s}\hat{\psi}(v\xi)\xi^s\hat{r}(\xi)||_{L^2}\\
			&\leq v^{s-1}||\xi^{-s}\hat{\psi}(\xi)||_{L^\infty}||r||_{H^s}\leq Cv^{s-1}||r||_{H^s},
		\end{align}
		the schwartz function $\hat{\psi}$ with $\hat{\psi}(0)=0$ lead to the latter bound in above inequality. In the similar method we estimate the $I_2$ as
		\begin{align}\label{eq:3.51}
			I_2\lesssim\int_0^\infty e^{-8tv^2}e^{-tz_2v^2}dv|||\varsigma-z_2|^{s-\frac{3}{2}}||_{L^p(v,\infty)}|||\varsigma-z|^{-1}||_{L^q(v,\infty)},
		\end{align}
		where$\frac{1}{p}+\frac{1}{q}=1$. By \cite{17} we get
		\begin{equation}\label{eq:3.52}
			|||\varsigma-z|^{-1}||_{L^q(v,\infty)}\leq C|v-\beta|^{\frac{1}{q}-1},
		\end{equation}
		and
		\begin{align}\nonumber
			&|||\varsigma-z_2|^{s-\frac{3}{2}}||_{L^p(v,\infty)}=\big(\int_v^\infty|u+iv|^{p(s-\frac{3}{2})}du\big)^{\frac{1}{p}}\\
			&=\big(\int_v^\infty(u^2+v^2)^{p\frac{2s-3}{4}}du\big)^{\frac{1}{p}}
			=v^{\frac{2s-3}{2}+\frac{1}{p}}\big(\int_v^\infty(u^2+1)^{p\frac{2s-3}{4}}du\big)^{\frac{1}{p}}.
		\end{align}
		So by(\ref{eq:3.51}) and using again (\ref{eq:3.49}), we obtain
		\begin{align}\nonumber
			I_2&\lesssim\int_0^\infty e^{-8tv^2}e^{-tz_2v^2}v^{\frac{2s-3}{2}+\frac{1}{p}}|v-\beta|^{\frac{1}{q}-1}dv\\
			&\lesssim\int_0^\infty e^{-8tv^2}e^{-tz_2v^2}v^{\frac{2s-3}{2}}dv\leq Ct^{\frac{1-2s}{4}}.
		\end{align}
		And by above estimates using standard facts, like dominated convergence we can proof the $J(L^\infty)\subset C^0$ and the readers can prove it for themselves.
\end{proof}
	
This lemma implies that    the integral equation (\ref{pieee}) has an unique solution, and we further  obtain the following result.
	\begin{proposition}\label{po32}There exists $\epsilon_0>0$ such that for $||r||_{H^s(\mathbb{R})\cap L^{2,s}(\mathbb{R})}<\epsilon$ there exist constants $T$ and $c$ such that for $t\ge T$ and for $z\in\Omega_{j2}\cup\Omega_{j5}$
		\begin{align}
			&M^{(3)}(z)=I+  M^{(3)}_1 z^{-1} +\mathcal{O}(z^{-2}),\label{eq:3.55}
		\end{align}
where
		\begin{align}
			&|M^{(3)}_1|\leq c||q_0||_{H^{1}(\mathbb{R})\cap L^{2,s}(\mathbb{R})}t^{-\frac{2s+1}{4}},\quad\mbox{for}\quad t\ge T. \label{eq:3.48}
		\end{align}
	\end{proposition}
	\begin{proof}In general we just consider the case when $j=2,i=1$. According to $M^{(3)}_1=\frac{1}{\pi}\int_\mathbb{C}M^{(3)}WdA$ we acquire that
$$|M^{(3)}_1|\leq\frac{||M^{(3)}||_{\infty}}{\pi}\int_{\Omega_{21}}|W|dA.$$
By  using
		\begin{align}\nonumber
			||e^{-8t(\alpha+6\beta z_2)uv}||_{L_u^q(v,\infty)}& =(8qt(\alpha+6\beta z_2)v)^{-\frac{1}{q}}e^{-8qt(\alpha+6\beta z_2)v^2}, \nonumber
		\end{align}
		We obtain
		\begin{align}\nonumber
			\int&_{\Omega_{21}}|X_1(\zeta)e^{2it\theta}|dA\leq||r||_{H^s\cap L^{2,1}}\int_0^\infty v^{s-1}||e^{-8t(\alpha+6\beta z_2)uv}||_{L_u^2(v,\infty)}dv\\
			&\lesssim t^{-\frac{1}{2}}\int_0^\infty v^{s-\frac{3}{2}}e^{-t(\alpha+6\beta z_2)v^2}dv||r||_{H^s\cap L^{2,1}}=C_st^{-\frac{2s+1}{4}}||r||_{H^s\cap L^{2,1}}.\nonumber
		\end{align}
		For $l=2$, we get
		\begin{align}\nonumber
			&\int_{\Omega_{21}}|X_2e^{2it\theta}|dA \leq||r||_{H^s\cap L^{2,1}}\int_0^\infty|||\varsigma-z_2|^{s-\frac{3}{2}}||_{L^p(v,\infty)}||e^{-8t(\alpha+6\beta z_2)uv}||_{L^q_u(v,\infty)}dv\\\label{eq:3.57}
			&\leq C||r||_{H^s\cap L^{2,1}}t^{-\frac{1}{q}}\int_0^\infty v^{\frac{2s-3}{2}+\frac{1}{p}-\frac{1}{q}}e^{-t(\alpha+6\beta z_2)v^2}dv\leq C_st^{-\frac{2s+1}{4}}||r||_{H^s\cap L^{2,1}}. \nonumber
		\end{align}
		Now we acquire (\ref{eq:3.55}) by  the Lipschitz continuous of Lemma \ref{le22}.
	\end{proof}
	
	\begin{theorem}\label{th31}
		Fix $s\in(\frac{1}{2},1]$ and let $q_0\in H^{1}(\mathbb{R})\cap L^{2,s}(\mathbb{R})\cap\mathcal{G}_0$. Then there exist constants $C(q_0)>0$ and $T(q_0)>0$ such that the solution of the Hirota equation (\ref{eq:1.1}) satisfies
		\begin{equation}
			||q(t,\cdot)||_{L^\infty(\mathbb{R})}\leq C(q_0)t^{-\frac{1}{2}},\,\,\mbox{for all} \,\, |t|\geq T(q_0).
		\end{equation}
		There are further constants $C_0>0,T_0>0$ and small $\epsilon>0$ such that for $||q_0||_{H^{1}(\mathbb{R})\cap L^{2,s}(\mathbb{R})}<\epsilon$, we can take $C(q_0)=C_0||q_0||_{H^{1}(\mathbb{R})\cap L^{2,s}(\mathbb{R})}$ and $T(u_0)=T_0$.
	\end{theorem}
	\begin{proof}
Recalling a series   of transformations (\ref{trns1}), (\ref{eq:3.15}) and  (\ref{trans2}),   we have
$$M(z) =M^{(3)}(z)M^{rhp}  (z) R^{(2)}(z)^{-1} \delta^{\sigma_3}.$$
Taking the limit $z\to \infty$ along $R^{(2)}(z)=I$ leads to
		$$M_1 =M^{(3)}_1 + \sum_{j=1}^2 \frac{M^{pc}_1}{\sqrt{(-1)^j8t(\alpha+6\beta z_j)}}+ \delta^{\sigma_3}.$$
		By using reconstruction formula (\ref{eq:2.28}), (\ref{eq:3.48})   for $t\ge T(s,\lambda_0)$ and a fixed $C=C(s,\lambda_0)$,  we have
		$$  |q(t,x)|\leq Ct^{-\frac{1}{2}},$$
		which  proves Theorem  \ref{th31}  for $q_0\in H^{1}(\mathbb{R})\cap L^{2,s}(\mathbb{R})\cap\mathcal{G}_0$.
	\end{proof}

	\section{The asymptotic stability of  the solitons }\label{sec4}
		
	\subsection{The B\"acklund transformation}
	
	  In this section we consider scattering data $\{r\equiv0,\{(z_k,c_k)_{k=1}^N\}$ for which the reflection coefficient vanishes identically correspond to $N$-soliton solution of (\ref{eq:1.1}).
Especially when $N=1$, the single soliton is given by (\ref{eq:1.3}), which  is a localized pulse with speed $v=-(12\beta\xi^2-4\beta\eta^2+4\alpha\eta^2)$ and maximum amplitude $2\eta$. Since $\mathcal{G}_1$ is an open subset of $L^1(\mathbb{R})$ and the soliton (\ref{eq:1.3}) belong to it. If the value of $\epsilon_0>0$ in the bound (\ref{eq:1.4}) is small enough, then the initial datum $q_0$ belong to $\mathcal{G}_1$. Notice also that the positive constant $\epsilon_0$ can be taken independent of $(\eta_0,x_0)$.
	
	In this section we consider the initial datum $q_0$ satisfying the bound $(\ref{eq:1.4})$ and the scattering datum associated with it belongs to the space $\mathcal{S}(1,1)$ defined in (\ref{eq:2.21}) is close to those of the soliton $q_{(\eta_0,\xi_0,\gamma_0))}(0,x)$ by Lemma \ref{le21} and Lemma \ref{le22}, we also obtain that when $q_0\in H^{1}(\mathbb{R})\cap L^{2,s}(\mathbb{R})$ implies $r\in H^s(\mathbb{R})\cap L^{2,1}(\mathbb{R})$. Furthermore, by the Lipschitz continuity of $q_0\to r$ and the fact that the soliton has $r\equiv0$,
	we have $||r||_{H^s(\mathbb{R})\cap L^{2,1}(\mathbb{R})} \leq C\epsilon$ with $C=C(\eta_0,\xi_0,\gamma_0)$ and the value of $\epsilon$ is given in (\ref{eq:1.4}).
	
	We define now a map
	\begin{equation}
		\mathcal{G}_1\times\mathbb{C}_+\times\mathbb{C}_*\ni(q_0,z_{s},c_1)\mapsto\tilde{q}_0\in\mathcal{G}_0,
	\end{equation}
	via the transformation
	\begin{equation}
		\tilde{r}(z):=r(z)\frac{z-z_{s}}{z-z^*_{s}}.
	\end{equation}
	From its definition, we acquire that $\tilde{r}\in H^s(\mathbb{R})\cap L^{2,1}(\mathbb{R})$ if $r\in H^s(\mathbb{R})\cap L^{2,1}(\mathbb{R})$ and there exist a constant $C>0$ such that $||\tilde{r}||_{H^s(\mathbb{R})\cap L^{2,1}(\mathbb{R})}\leq C||r||_{H^s(\mathbb{R})\cap L^{2,1}(\mathbb{R})}$. Then we define $\tilde{q}\in\mathcal{G}_0\,\cap\,H^{1}\,\cap\,L^{2,s}(\mathbb{R})$ by the reconstruction formula (\ref{eq:2.28}) with the corresponding RH problem is solved for the scattering datum in $\mathcal{S}(1,0)=\{\tilde{r}\in H^s(\mathbb{R})\cap L^{2,1}(\mathbb{R})\}$. By Lemma \ref{le24} we know that $\tilde{q_0}\in\mathcal{G}_0\,\cap H^s(\mathbb{R})\cap L^{2,1}(\mathbb{R})$ with norm $||\tilde{q}_0||_{H^1(\mathbb{R})\cap L^{2,s}(\mathbb{R})}\leq C||\tilde{r}||_{H^s(\mathbb{R})\cap L^{2,1}(\mathbb{R})}\leq C\epsilon$.
	
	The BT  technique  is a method that can derive multisoliton solution starting from  a  trivial seed solution  in a purely algebraic procedure for  a  integrable equation. Main feature of BT is that the Lax pair associated with the integrable equation  remains covariant under the gauge transformation.
	
	\begin{lemma}\label{le41}
 Define the BT that  generates a new solution of the Hirota equation as
		\begin{equation}\label{eq:4.3}
			\tilde{q}=q-\mathbf{B},\quad\mathbf{B}:=2i(z_{s}-z_{s}^*)\frac{f_1f^*_2}{|f_1|^2+|f_2|^2},
		\end{equation}
		where
		$$f_1:=e^{-ixz_s}m_{11}(t,x,z_{s})-\frac{c_1m_{12}(t,x;z_s)e^{ixz_{s}+i(4\beta z_{s}^3+2\alpha z_{s}^2)t}}{2i\mbox{Im}(z_s)},$$
		$$f_2:=e^{-ixz_s}m_{21}(t,x,z_{s})-\frac{c_1m_{22}(t,x,z_s)e^{ixz_{s}+i(4\beta z_{s}^3+2\alpha z_s^2)t}}{2i\mbox{Im}(z_s)}.$$
		The proof will be given in the  \ref{AB}.
	\end{lemma}

	\begin{remark}\label{re41}
 The soliton (\ref{eq:1.3}) can be  recovered with the BT  (\ref{eq:4.3})  by taking
		\begin{align}
	&	r=0,\quad z_s =\xi +i\eta, \quad  \tilde{q}=0,\nonumber\\
&f_1=e^{-ixz_s},\quad\mbox{and}\quad f_2=-\frac{c_1}{2i\eta }e^{ixz_{s}+i(4\beta z_{s}^3+2\alpha z_{s}^2)t}.\nonumber
		\end{align}
	\end{remark}

		\begin{lemma}\label{le34} Let $z_s=\alpha_1+i\beta_1$ with $\beta_1>0$. There is $\epsilon_0$ sufficiently small such that for $||q_0||_{H^{1}(\mathbb{R})\,\cap\,L^{2,s}(\mathbb{R})}<\epsilon_0$, there is a constant $C$ such that
		\begin{align}\nonumber
			&|I-V_+(z_s)|\leq Ce^{-t8\beta_1^2}||q_0||_{H^{1,s}(\mathbb{R})},\,\,\mbox{if}\,\, z_s\in\Omega_{j1}+z_j,\\\label{eq:3.52}
			&|I-U^{-1}_R(z_s)|\leq Ce^{-t8\beta_1^2}||q_0||_{H^{1,s}(\mathbb{R})},\,\,\mbox{if}\,\, z_s\in\Omega_{j3}+z_j.
		\end{align}
	\end{lemma}
	\begin{proof} For $j=2,i=1,3$ from (\ref{eq:3.13}) we obtain that
$$||R_i||_{L^\infty(\Omega_{2i}+z_2)}\leq C'||r||_{H^s(\mathbb{R})\cap L^{2,2}(\mathbb{R})}
		\leq C||q_0||_{L^{2,s}(\mathbb{R})\cap H^2(\mathbb{R})}.$$
 If $z_s\in\Omega_{21}+z_2$ we have $\alpha_1-z_2\ge\beta_1$ and so
 $$|e^{-2it\theta}|\lesssim e^{-8t(\alpha-z_2)\beta_1}\lesssim e^{-t8\beta_1^2}.$$
 If $z_s\in\Omega_{23}+z_2,$
we have similarly $|e^{2it\theta}|\leq e^{-t8\beta_1^2}$ which yield (\ref{eq:3.52}).\end{proof}
	By Theorem \ref{th31}, we know that there exist constants $C_0$ and $T>0$ such that for all $t>T$, we have
	$$||\tilde{q}(t,\cdot)||_{L^\infty(\mathbb{R})}\leq C_0\epsilon|t|^{-\frac{1}{2}}.$$
	Since there is a constant $C>0$ such that $||\tilde{q}_0||_{H^1(\mathbb{R})\,\cap\, L^{2,s}(\mathbb{R})}\leq C\epsilon$.
	In order to prove Theorem \ref{th1} we need to focus only on $\mathbf{B}$, we focusing on $t\gg1$ we know that
	
	\begin{lemma}\label{le33} Suppose that $||q_0||_{H^{1}(\mathbb{R})\,\cap\, L^{2,s}(\mathbb{R})}<\epsilon_0$ and $z_s\in\mathbb{C}_+$, then there are a $\epsilon_0>0$, a $c>0$ and a $T>0$ such that
		\begin{equation}
			|I-M^{(3)}(z_s)|\leq ct^{-\frac{2s+1}{4}}||q_0||_{H^{1}(\mathbb{R})\cap L^{2,s}(\mathbb{R})},\quad\mbox{for} \quad t\ge T.
		\end{equation}
	\end{lemma}
	\begin{proof} Notice the function $M^{(3)}(z)$ satisfying the inequality with
		$$|M^{(3)}-I|\leq\frac{||M^{(3)}||_{\infty}}{\pi}\sum_i\int_{\Omega_{2i}}\frac{|W|}{|\zeta-z_s|}dA,$$
		and similarly as in Proposition \ref{po32} we just consider the case when $j=2,i=1$. We set $z_s=\alpha_1+i\beta_1$ and for $l=1,3$ and just like the case in the (\ref{eq:3.45}) and (\ref{eq:3.46}) we obtain
		\begin{align}
			\int_{\Omega_{21}}\frac{|X_l(\varsigma)e^{2it\theta}|}{|\varsigma-z_2|}dA\leq||r||_{H^s(\mathbb{R})}[A_1+A_2],
		\end{align}
		where
		$$A_1:=\int_0^{\frac{\rho_1}{2}}v^{s-1}\Big|\Big|\frac{e^{-8t(\alpha+6\beta z_2)uv}}{|\varsigma-z_2|}\Big|\Big|_{L_u^2(v,\infty)}dv,$$
		and
		$$A_2:=\int_\frac{\rho_1}{2}^{\infty}v^{s-1}\Big|\Big|\frac{e^{-8t(\alpha+6\beta z_2)uv}}{|\varsigma-z_2|}\Big|\Big|_{L_u^2(v,\infty)}dv.$$
		We get
		\begin{align}\nonumber
			A_1&=\int_0^{\frac{\rho_1}{2}}v^{s-1}\Big|\Big|\frac{e^{-8t(\alpha+6\beta z_2)uv}}{\sqrt{(u+z_2-\alpha_1)^2+(v-\beta_1)^2}}\Big|\Big|_{L_u^2(v,\infty)}dv\\
			&\leq C'(\beta_1)\int_0^{\frac{\rho_1}{2}}v^{s-1}\Big|\Big|e^{-8t(\alpha+6\beta z_2)uv}\Big|\Big|_{L_u^2(v,\infty)}dv\leq C(s,\beta_1)t^{-\frac{4s+1}{8}}.
		\end{align}
		By using (\ref{eq:3.52}) and consider the case when $t\ge1$ and $e^{-8tv^2}\leq e^{-t\gamma_1^2}e^{-4v^2}$ for $v\ge\frac{\beta_1}{2}$ we acquire the bound for $A_2$ is
		\begin{align}\nonumber
			A_2&\leq\int_\frac{\rho_1}{2}^\infty e^{-8t(\alpha+6\beta z_2)v^2}v^{s-1}|||\varsigma-z_2|^{-1}||_{L_u^2(v,\infty)}dv\\\nonumber
			&\leq C\int_\frac{\rho_1}{2}^\infty e^{-8t(\alpha+6\beta z_2)v^2}v^{s-1}|v-\beta_1|^{-\frac{1}{2}}dv\\
			&\leq Ce^{-t\beta_1^2}\int_0^\infty e^{-4(\alpha+6\beta z_2)v^2}v^{s-1}|v-\beta_1|^{-\frac{1}{2}}dv\leq C'e^{-t\beta_1^2}.
		\end{align}
		For $l=2$ we similarly have
		\begin{align}
			\int_{\Omega_{21}}\frac{|X_2(\varsigma)e^{2it\theta}|}{|\varsigma-z_2|}dA\leq||r||_{H^s(\mathbb{R})}[B_1+B_2],
		\end{align}
		where
		$$B_1:=\int_0^{\frac{\rho_1}{2}}\int_v^\infty |\varsigma-z_2|^{s-\frac{3}{2}}\frac{e^{-8t(\alpha+6\beta z_2)uv}}{|\varsigma-z_2|}dv,$$
		and
		$$B_2:=\int_{\frac{\rho_1}{2}}^\infty\int_v^\infty |\varsigma-z_2|^{s-\frac{3}{2}}\frac{e^{-8t(\alpha+6\beta z_2)uv}}{|\varsigma-z_2|}dv.$$
		Then we obtain $B_1\leq C(\beta,s)t^{-\frac{2s+1}{8}}$   and $|\varsigma-z_2|\ge\beta_1/2$.  And when $t\ge1$ we have
		\begin{align}\nonumber
			B_2&\leq\int_{\frac{\rho_1}{2}}^\infty e^{-8t(\alpha+6\beta z_2)v^2}|||\varsigma-z_2|^{s-\frac{3}{2}}||_{L^p(v,\infty)}|||\varsigma-z_2|^{-1}||_{L^q(v,\infty)}dv\\
			&\leq Ce^{-t\beta_1^2/2}\int_0^\infty  e^{-4(\alpha+6\beta z_2)v^2}v^{\frac{2s-3}{2}+\frac{1}{p}}|v-\beta_1|^{\frac{1}{q}-1}dv\leq C_s e^{-ty_1^2/2}.
		\end{align}
	\end{proof}
	\begin{lemma}\label{le42} Fix $\lambda_0>0$. Then there is a $C>0$ and a $T>0$ such that for $||\tilde{r}||_{H^s(\mathbb{R})}<\lambda_0$ we have for $t\ge T$
		\begin{align}\nonumber
			&\big|m_{11}(t,x,z_s)-\delta(z_s)\big|+\big|m_{22}(t,x,z_s)-\delta^{-1}(z_s)\big|\\
			&\qquad\qquad\qquad\qquad\leq C||\tilde{r}||_{H^s(\mathbb{R})}t^{-\frac{1}{2}}(||\tilde{r}||_{H^s(\mathbb{R})}+t^{-\frac{2s-1}{8}}),\\\nonumber
			&\big|m_{12}(t,x,z_s)-\frac{\delta^{-1}(z_s)\beta_{12}}{\sqrt{(-1)^j8t(\alpha+6\beta z_j)}(z_s-z_j)}\big|
			\\\nonumber
			&\qquad\qquad\qquad\qquad+|m_{22}(t,x,z_s)-\frac{\delta(z_s)\beta_{21}}{\sqrt{(-1)^j8t(\alpha+6\beta z_j)}(z_s-z_j)}\big|\\
			&\qquad\qquad\qquad\qquad\leq C(||\tilde{r}||_{H^s(\mathbb{R})}t^{-\frac{2s+1}{8}}).
		\end{align}
	\end{lemma}
	\begin{proof} We focus on $j=2$  and recall the Lemma \ref{le33} we have the expansion of function $M^{(3)}=I+\mathcal{O}(||\tilde{r}||_{H^s(\mathbb{R})}t^{-\frac{2s+1}{8}})$ and Lemma \ref{le34} we get similar expansions for $V_+(z_s)$ and $U_R(z_s)$. We furthermore know by the property of $|\delta^{\pm}|\leq1+\rho^2$ for $\rho=||\tilde{r}||_{L^\infty(\mathbb{R})}$ and  we acquire $|\beta_{12}|+|\beta_{21}|<C||\tilde{r}||_{L^\infty(\mathbb{R})}$. From \ref{AA} we have the expansion
		$$M^{pc}(\zeta_2)=I+\frac{M^{pc}_1}{\sqrt{8t(\alpha+6\beta z_2)}(z_s-z_2)}+\mathcal{O}(||\tilde{r}||_{H^s(\mathbb{R})}t^{-1}).$$
 These observations yield Lemma \ref{le42}.\end{proof}
	
	\subsection{The proof of main result}
	Now we start to analyze the term $\mathbf{B}$ in (\ref{eq:4.3}). Consider the following inequalities:
	\begin{align}\label{eq:4.7}
		&\big|e^{-ixz_s}m_{11}(t,x,z_s)\big|>10\big|\frac{c_1m_{12}(t,x;z_s)e^{ixz_s+i(2\alpha z_s^2+4\beta z_s^3)t}}{2i\mbox{Im}(z_s)}\big|,\\\label{eq:4.8}
		&10\big|e^{-ixz_1}m_{21}(t,x;z_s)\big|<\big|\frac{c_1m_{22}(t,x,z_s)e^{ixz_s+i(2\alpha z_s^2+4\beta z_s^3)t}}{2i\mbox{Im}(z_s)}\big|.
	\end{align}
	
	\begin{lemma}\label{le43}Given $\epsilon>0$ small, there exist $T(\epsilon_0)>0$ and $C>0$ such that if $||\tilde{r}||_{H^s(\mathbb{R})\,\cap\,L^{2,s}(\mathbb{R})}<\epsilon_0$ and if $(t,x)$ is such that at least one of (\ref{eq:4.7}) and (\ref{eq:4.8}) is false, then we have $|\mathbf{B}|<Ct^{-\frac{1}{2}}\epsilon$ for $t\geq T(\epsilon_0)$.
	\end{lemma}
	\begin{proof} For the $\epsilon$ of (\ref{eq:1.4}) and $\rho=||\tilde{r}||_{L^\infty(\mathbb{R})}$. We are only consider the case when $t$ is large by assuming that for $(t,x)$ inequality(\ref{eq:4.7}) is false, Lemma \ref{le42} implies for $t\ge T$
		\begin{align}\nonumber
			|m_{12}|&\leq(1+\rho^2)|k_1|t^{-\frac{1}{2}}+C\epsilon t^{-\frac{2s+1}{8}}\\\nonumber
			&\leq t^{-\frac{1}{2}}\epsilon K\big(\frac{1}{2}(1+\rho^2)^{-1}-Ct^{\frac{1-2s}{8}}-C\epsilon\big)\\\label{eq:4.9}
			&\leq t^{-\frac{1}{2}}\epsilon K|m_{22}|,
		\end{align}
		for a fixed and sufficiently large constant $K$. Then, if (\ref{eq:4.7}) is false and $t\ge T$, both terms in (\ref{eq:4.7}) are bounded from above by
		$$\Big|\frac{c_1m_{22}(t,x;z_s)e^{ixz_s+i(2\alpha z_s^2+4\beta z_s^3)t}}{2i\mbox{Im}(z_s)}\Big|.$$
		For $t\ge T$ by the same argument of (\ref{eq:4.9}) we have also
		\begin{equation}\label{eq:4.10}
			\big|e^{-ixz_1}m_{21}(x,t;z_s)\big|\leq t^{-\frac{1}{2}}\epsilon K|e^{-ixz_s}m_{11}(x,t;z_s)\big|.
		\end{equation}
		We conclude that for $t\ge T$ and if $(t,x)$ is in the domain where (\ref{eq:4.9}) is false, we have for some fixed $K$
		\begin{equation}\label{eq:4.11}
			|\mathbf{B}|\leq K\frac{\big|m_{12}e^{ixz_s-i(2\alpha z_s^2+4\beta z_s^3)t}m^*_{22}e^{-ixz^*_s-i(2\alpha z^{*2}_s+4\beta z_s^{*3})t}\big|}{|m_{22}e^{ixz_s+i(2\alpha z_s^2+4\beta z_s^3)t}|^2}
			=K\frac{|m_{12}|}{|m_{22}|}\leq\frac{CK}{\sqrt{t}}\epsilon.
		\end{equation}
		Now we assume that $(t,x)$ is such that (\ref{eq:4.7}) is true. Notice that by (\ref{eq:4.7}) and (\ref{eq:4.10}) we have for a fixed $K$
		\begin{equation}\label{eq:4.12}
			\frac{|f_1e^{ixz_s^*}m_{21}^*|}{||f||^2}\leq K\frac{\big|e^{-ixz_s}{m_{11}e^{-ixz_s^*}}m_{21}^*|}{|e^{-ixz_s}m_{11}|^2}=K\frac{|m_{21}|}{|m_{11}|}\leq\frac{CK}{\sqrt{t}}\epsilon.
		\end{equation}
		Consider the suppose that $(t,x)$ is such that (\ref{eq:4.7}) and (\ref{eq:4.8}) are not both true, we assume now that (\ref{eq:4.7}) is true and (\ref{eq:4.8}) is false. Then by (\ref{eq:4.12}), for a fixed $K$
		\begin{equation}
			|\mathbf{B}|\leq4|\mbox{Im} z_s|\frac{f_1f_2}{||f||^2}\leq K\frac{|f_1e^{ixz_s^*}m_{21}^*|}{||f||^2}\leq\frac{CK}{\sqrt{t}}\epsilon.
		\end{equation}
		This prove the Lemma \ref{le42} for values of $(t,x)$ for which (\ref{eq:4.7}) and (\ref{eq:4.8}) are true.
		\end{proof}
			\begin{lemma}\label{le35} Fix $\rho_0>0$ and let $\rho:=||r||_{L^\infty(\mathbb{R})}$ and assume $\rho<\rho_0$, then for fix $z_s=\alpha_1+i\beta_1$ with $\beta_1>0$ there exists a constant $C$ independent from $z_j$ such that
			\begin{align}\nonumber
				&|\delta(z_s)-\nu(z_s)|\leq C||r||_{L^2}^2,\\\label{eq:3.53}
				&\mbox{where}\,\,\,\,\,\nu(z_s):=\exp\Big(\frac{1}{2\pi i}\int_{-\infty}^{\alpha_1}\frac{\log(1+|r(\varsigma)|^2}{\varsigma-z_s}d\varsigma\Big),
			\end{align}
			and fix $K>0$ then for $|z_j-\alpha_1|\leq K/\sqrt{t}$ there exists a constant $C$ such that
			\begin{equation}\label{eq:3.54}
				|\delta(z_s)-\nu(z_s)|\leq \frac{C}{\sqrt{t}\beta_1}\log(1+\rho^2).
			\end{equation}
		\end{lemma}
		\begin{proof}According to the Proposition \ref{po32} and for a fixed $c$ and for
			\begin{equation}\nonumber
				\gamma(z):=\frac{1}{2\pi i}\int_{z_1}^{z_2}\frac{\log(1+|r(s)|^2)}{s-z}ds,
			\end{equation}
			rewrite as $\delta(z)=\delta(z)^{\gamma(z)}$ and we obtain
			\begin{equation}
				\Big|\gamma(z_s)-\frac{1}{2\pi i}\int_{z_1}^{\alpha_1}\frac{\log(1+|r(s)|^2)}{s-z_s}ds\Big|=\frac{1}{2\pi}\Big|\int_{\alpha_1}^{z_2}\frac{\log(1+r(s)|^2}{s-\alpha_1-i\beta_1}ds\Big|\leq\frac{c}{\beta_1}||r||_{L^2}^2.
			\end{equation}
			This yield (\ref{eq:3.53}) since the bound $|\delta(z)|\leq(1+\rho^2)$ is independent from $z_2$. Similarly (\ref{eq:3.54}) follows from
			\begin{equation}
				\Big|\gamma(z_s)-\frac{1}{2\pi i}\int_{z_1}^{\alpha_1}\frac{\log(1+r(s)|^2)}{s-z_s}ds\Big|=\frac{1}{2\pi}\Big|\int_{\alpha_1}^{z_2}\frac{\log(1+|r(s)|^2)}{s-\alpha_1-i\beta_1}ds\Big|.
			\end{equation}
			This yield Lemma \ref{le35}.\end{proof}
		
		 Then, We assume that (\ref{eq:4.7}) and by (\ref{eq:4.8}) are true, from the inequality (\ref{eq:4.11}) and (\ref{eq:4.12}), up to terms bounded by $Ct^{-\frac{1}{2}}\epsilon$, what is left is the analysis of
		\begin{equation}\label{eq:4.14}
			-2i\frac{e^{-ixz_s}m_{11}c_1^*m^*_{22}e^{-ixz_s^*-(2\alpha z_s^{*2}+4\beta z_s^{*3})t}}{||f||^2}.
		\end{equation}
		Set now
		\begin{equation}\nonumber
			b^2=\big|e^{-ixz_s}m_{11}\big|^2+\big|\frac{c_1m_{22}e^{ixz_s+(2\alpha z_s^2+4\beta z_s^3)t}}{2i\mbox{Im}(z_s)}\big|^2,
		\end{equation}
		and expand
		\begin{equation}\nonumber
			||f||^2=b^2\big(1+\mathcal{O}(b^{-1}\big|c_1m_{12}e^{ixz_s+(2\alpha z_s^2+4\beta z_s^3)t}\big|)+\mathcal{O}(b^{-1}|m_{21}(t,x;z_1)e^{-ixz_1})\big).
		\end{equation}
		Then the quantity in (4.14) is of the form
		\begin{align}\nonumber
			-2i&e^{-ixz_s}m_{11}\frac{c_1^*m_{22}^*e^{-ixz_s^*-(2\alpha z_s^{*2}+4\beta z_s^{*3})t}}{b^2}\\
			&\times\big(1+\mathcal{O}(b^{-1}\big|m_{12}e^{ixz_s+(2\alpha z_s^2+4\beta z_s^3)t}\big|)+\mathcal{O}(b^{-1}|m_{21}e^{-ixz_s})\big).
		\end{align}
		We claim that the quality in (\ref{eq:4.14}) equals
		\begin{equation}\label{eq:4.16}
			-2i\frac{e^{-ixz_1}\delta(z_s)(\delta(z_s))^{-1}c_1^*e^{-ixz_s^*-(2\alpha z_s^{*2}+4\beta z_s^{*3})t}}{b^2}(1+\mathcal{O}(\epsilon t^{-\frac{1}{2}})).
		\end{equation}
		To prove this claim, since $m_{ii}=\delta^{-(-1)^i}(z_1)+\mathcal{O}(\epsilon^{-\frac{1}{2})}$ and $|\delta^{\pm1}(z_1)|\ge\langle\rho\rangle^{-2}$, we have
		\begin{equation}\nonumber
			b^2=(\big|\delta(z_s)e^{-ixz_s}\big|^2
			+\big|\frac{c_1\delta(z_s)^{-1}e^{ixz_s+(2\alpha z_s^2+4\beta z_s^3)t}}{2i\mbox{Im}(z_s)}\big|^2)(1+\mathcal{O}(\epsilon t^{-\frac{1}{2}})).
		\end{equation}
		We have $\mathcal{O}(b^{-1}\big|c_1m_{12}e^{ixz_s+(2\alpha z_s^2+4\beta z_s^3)t}\big|)
		=\mathcal{O}(\epsilon t^{-\frac{1}{2}})$ by
		$$b^{-1}\big|c_1m_{12}e^{ixz_s+(2\alpha z_s^2+4\beta z_s^3)t}
		\big|\leq\frac{|m_{12}e^{ixz_s+(2\alpha z_s^2+4\beta z_s^3)t}|}{|m_{22}e^{ixz_s+(2\alpha z_s^2+4\beta z_s^3)t}|}=\frac{|m_{12}|}{|m_{22}|}\leq C\epsilon t^{-\frac{1}{2}}.$$
		We have $\mathcal{O}(b^{-1}\big|m_{21}e^{ixz_s^*}\big|)=\mathcal{O}(\epsilon t^{-\frac{1}{2}})$ by
		$$b^{-1}\big|m_{21}e^{ixz_s}\big|\leq\frac{|m_{12}e^{-ixz_s}|}{|m_{11}e^{-ixz_s}|}=\frac{|m_{21}|}{|m_{11}|}\leq C\epsilon t^{-\frac{1}{2}}.$$
		hence (\ref{eq:4.16}) is proved.
		Consider the term in (\ref{eq:4.16}), now for $z_s=\xi+i\eta$ and $\nu(z_s)$ defined in (\ref{eq:3.6}) and inserting trivial factors $\nu/\nu=1$ and $\nu^*/\nu^*=1$, the expression in (\ref{eq:4.16}) equals
		\begin{equation}\label{eq:4.17}
			\frac{4\eta e^{-ixz_s-i(2\alpha z_s^2+4\beta z_s^3)t}e^{-ixz_s^*-i(2\alpha z_s^{*2}+4\beta z_s^{*3})t}\frac{\delta(z_s)}{\nu(z_s)}\frac{\nu^*(z_s)}{\delta^*(z_s)}\frac{\nu(z_s)}{\nu(z_s)}}
			{\tilde{b}^2},
		\end{equation}
		where
		\begin{align}\nonumber
			\tilde{b}^2=\big|&e^{-2\eta x-24\beta\xi^2t+8\beta\eta^3t-8\alpha\eta\xi t}\big|^2|\frac{\delta(z_s)}{\nu(z_s)}||\nu(z_s)|\\\nonumber
			&\qquad\qquad\qquad+\big|e^{2\eta x+24\beta\xi^2t-8\beta\eta^3t+8\alpha\eta\xi t}\big|^2|\frac{\nu(z_s)}{\delta(z_s)}||\nu(z_s)|^{-1}.
		\end{align}
		Now fix a constant $\kappa>0$,Then (\ref{eq:4.17}) differs from the soliton solution
		\begin{align}\nonumber
			2&\eta e^{2i(-\xi x-4\beta\xi^3t-2\alpha\xi^2t+12\beta\xi\eta^2t+2\alpha\eta^2t)+i\gamma}\\\label{eq:4.18}
			&\qquad\qquad\times\mbox{sech}(-2\eta x-24\beta\xi^2t+8\beta\eta^3t-8\alpha\eta\xi t+\log(|\nu(z_s)|)),
		\end{align}
		by less than $c\kappa t^{-\frac{1}{2}}\epsilon$ by the sum of the following two error terms: notice the  (\ref{eq:4.17}) and (\ref{eq:4.18}) can be bounded, up to a constant factor $C=C(\xi,\eta,\alpha,\beta)$, by the sum of the following two error terms:
		\begin{equation}\label{eq:4.19}
			\frac{\big|\frac{\delta(z_s)}{\nu(z_s)}\frac{\nu^*(z_s)}{\delta^*(z_s)}-1\big|}{e^{8(3\beta|z_2-\xi|^2t-\beta\eta^3t+\alpha\eta|z_2-\xi| t)}(1+||\tilde{r}||_{L^\infty(\mathbb{R})}^2)^{-1}},\end{equation}
		and
		\begin{align}\nonumber
			\big|&\mbox{sech}(-8(3\beta|z_2-\xi|^2t-\beta\eta^3t+\alpha\eta|z_2-\xi| t+\log(|\nu(z_s)|))\\\label{eq:4.20}
			& - \mbox{sech}( -8(3\beta|z_2-\xi|^2t-\beta\eta^3t+\alpha\eta|z_2-\xi| t+\log(|\nu(z_s)|)+\log\big(\frac{|\delta(z_s)|}{|\nu(z_s)|}\big))\big|.
		\end{align}
		At first we bounded (\ref{eq:4.19}) and for $|z_2-\xi|\ge\kappa t^{-\frac{1}{2}}$ formula (\ref{eq:4.19}) is bounded by $Ce^{8(3\beta|\sqrt{t}\kappa}e^{-\beta\eta^3t+\alpha\eta\kappa\sqrt{t})}.$ For
		$|z_2-\xi|\leq\kappa t^{-\frac{1}{2}}$ we bounded (\ref{eq:4.19}) by (\ref{eq:3.54})
		$$(1+||\tilde{r}||_{L^\infty(\mathbb{R})}^2)\Big|\frac{\delta(z_s)}{\nu(z_s)}\frac{\nu^*(z_s)}{\delta^*(x_s)}-1\Big|\leq4\frac{C}{\sqrt{t}}(1+||\tilde{r}||_{L_\infty^2(\mathbb{R})})\leq Kt^{-\frac{1}{2}}\epsilon^2.$$
		According to Lagrange Theorem, (\ref{eq:4.20}) is bounded
		\begin{align}\nonumber
			\mbox{sech}( -8(3\beta|z_2-\xi|^2t&-\beta\eta^3t+\alpha\eta|z_2-\xi| t+\log(|\nu(z_s)|)\\\nonumber
			&+\log\big(\frac{|\delta(z_s)|}{|\nu(z_s)|}\big))\Big|\log\Big(\frac{|\delta(z_s)|}{|\nu(z_s)|}\Big)\Big|,
		\end{align}
		for some $c\in(0,1)$. This satisfies bounds similar to those satisfied by (\ref{eq:4.19}).
		
		We consider initial potential $q_0\in H^{1}(\mathbb{R})\,\cap\,L^{2,s}(\mathbb{R})$ we need to show that when one of (\ref{eq:4.7}) and (\ref{eq:4.8}) is false to complete the proof of Theorem \ref{th1}, then the function in (\ref{eq:4.18}) is $\mathcal{O}(\epsilon t^{-\frac{1}{2}})$. By Lemma \ref{le42} the fact that (\ref{eq:4.7}) and (\ref{eq:4.8}) false means that for a fixed $C=C(\rho_0)>0$ we have
		$$|e^{-2ixz_s-2i(\alpha z_s^{*2}+2\beta z_s^{*3})t}|\leq C\epsilon t^{-\frac{1}{2}},$$
		and
		$$|e^{2ixz^*_s+2i(\alpha z_s^{2}+2\beta z_s^{3})t}|\leq C\epsilon t^{-\frac{1}{2}}.$$
		Any of these yield our claim that the function in (\ref{eq:4.18}) is $\mathcal{O}(\epsilon t^{-\frac{1}{2})}$
		Then we complete the proof of Theorem \ref{th1} for $q_0\in H^{1}(\mathbb{R})\,\cap\,L^{2,s}(\mathbb{R})$ and we notice that when $|t|\ge T(\epsilon_0)$ the soliton in formula (\ref{eq:1.3}) is given by formula (\ref{eq:4.18}).
	
In the end of this paper of we will explain the ground states $q_{\xi,\eta,\gamma_{\pm},\alpha,\beta}(t,x-x_{\pm})$ in the statement of Theorem \ref{th1} are in general distinct. The + ground state has been computed explicitly in (\ref{eq:4.18}).
	\\
	\begin{lemma}\label{le48} Considering $t<-T(\lambda_0)$ ground state is given by formula (\ref{eq:4.20}) but with $\nu(z_s)$ replaced by
		\begin{equation}\label{eq:4.21}
			\Lambda(z_s)=\exp\Big(\frac{1}{2\pi i}\int_{\alpha_1}^{\infty}\frac{(\log(1+|r(s)|^2)}{s-z_s}ds)\Big).
	\end{equation}
	\end{lemma}
	\begin{proof} Notice that if $q(t,x)$ solves (\ref{eq:1.1}) then $u(t,x):=q^*(-t,x)$ solves the Hirota equation with initial value $q^*_0(x)$, and if $(r(z),z_s,c_1)$ are the spectral data of $q_0\in\mathcal{G}_1$, then we have $q^*_0\in\mathcal{G}_1$ with spectral data $(r^*(-z),-z_s^*,-c_1^*)$. According to (\ref{eq:4.18}) we obtain when $t\to-\infty$
		\begin{align}\nonumber
			u(-t,x)&\sim-2\eta e^{2i(\xi x+4\beta\xi^3t+2\alpha\xi^2t-12\beta\xi\eta^2t-2\alpha\eta^2t)-i\gamma}\\\nonumber
			&\times\mbox{sech}(-2\eta x-24\beta\xi^2t+8\beta\eta^3t-8\alpha\eta\xi t+\log(|\Lambda(z_s)|)).
		\end{align}
		And the complex conjugate of $\Lambda(z_s)$ can be written as
		\begin{equation}\nonumber
			\Lambda^*(z_s)=\exp\Big(\frac{1}{2\pi i}\int_{-\infty}^{\xi}\frac{(\log(1+|r(s)|^2)}{s+\xi-i\eta}ds)\Big).
		\end{equation}
	Then (\ref{eq:4.21}) is true and using $q(t,x)=u^*(-t,x)$ and so taking the complex conjugate of the above formula, we obtain for $t\to-\infty$
		\begin{align}\nonumber
		q(t,x)&\sim2\eta e^{2i(-\xi x-4\beta\xi^3t-2\alpha\xi^2t+12\beta\xi\eta^2t+2\alpha\eta^2t)+i\gamma}\\\nonumber
			&\times\mbox{sech}(-2\eta x-24\beta\xi^2t+8\beta\eta^3t-8\alpha\eta\xi t+\log(|\Lambda(z_s)|)).
		\end{align}
		thus completing the proof of Lemma \ref{le48}.
	\end{proof}
	\clearpage

	\appendix
	\section{A parabolic cylinder model}\label{AA}

	Here we describe the solution of the parabolic cylinder model problem \cite{Cuccagna2013}. Let $\Sigma^{pc}=\cup_{n=1}^4\Sigma_j$, where $\Sigma_j$ denotes the complex contour
	
\begin{equation}
		\Sigma_j=\{\zeta\in\mathbb{C}|\arg\zeta=\frac{2j-1}{4}\pi\},\quad j=1,\dots,4.
	\end{equation}
	Denote $\Omega_j,j=1,\dots,6$ be the six maximally connected open sectors in $\mathbb{C}\setminus(\Sigma^{pc}\cup\mathbb{R})$ where
its labelled sequentially as one encircles the origin in a counterclockwise fashion. Finally, fix $r_0\in\mathbb{C}$ and let
	\begin{equation}
		\kappa :=-\frac{1}{2\pi}\log(1+|r_0|^2).
	\end{equation}
	Then consider the following RH problem.

	\begin{problem}\label{PC}
Fix $r_0\in\mathbb{C}$, find an analytic function $M^{pc}(\cdot):\mathbb{C}\setminus\Sigma^{pc}\to SL_{2}(\mathbb{C})$
		\begin{enumerate}
			\item $M^{pc}(\zeta)=I+ \zeta^{-1}  M^{pc}_1 +\mathcal{O}(\zeta^{-2})$ uniformly as $\zeta\to\infty$.
			\item For $\zeta\in\Sigma^{pc}$, the continuous values $M^{pc}_{\pm}(\zeta)$ satisfy the jump relation $$M^{pc}_+(\zeta)=M^{pc}_-(\zeta)V^{pc}(\zeta),$$
			where
		\end{enumerate}
		\begin{equation}
			V^{pc}(\zeta  )=\begin{cases}
				\begin{pmatrix}
					1&0\\
					r_0\zeta^{-2i\kappa}e^{i\zeta^2/2}&1
				\end{pmatrix},\quad\,\,\,\,\quad\qquad\mbox{arg}\zeta=\frac{\pi}{4},\\
				\\
				\begin{pmatrix}
					1&r_0^*\zeta^{2i\kappa}e^{-i\zeta^2/2}\\
					0&1
				\end{pmatrix},\quad\,\,\quad\,\qquad\mbox{arg}\zeta=-\frac{\pi}{4},\\
				\\
				\begin{pmatrix}
					1&\frac{r_0^*}{1+|r_0|^2}\zeta^{2i\kappa}e^{-\zeta^2/2}\\
					0&1
				\end{pmatrix},\quad\qquad\mbox{arg}\zeta=\frac{3\pi}{4},\\
				\\
				\begin{pmatrix}
					1&0\\
					\frac{r_0}{1+|r_0|^2}\zeta^{-2i\kappa}e^{i\zeta^2/2}&1
				\end{pmatrix},\quad\qquad\mbox{arg}\zeta=-\frac{\pi}{4},
			\end{cases}
		\end{equation}
See Figure \ref{fgsd}.
	\end{problem}
	
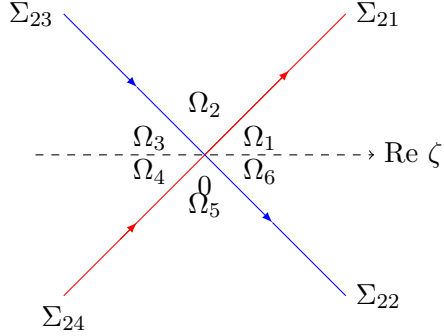
\begin{figure}\label{F4}
		\begin{center}
			\begin{tikzpicture}[scale=0.75]
				[node distance=1cm]
				%\draw[->](-5,0)--(5,0) node[right] {$Rez$};
				\draw [dashed,-> ](-3,0)--(3,0)node[right]{Re $\zeta$};	
				\draw[-latex,red](0,0)--(1.5,1.5) ;
				\draw[red](2.5,2.5)--(1.5,1.5) ;
				\node[right] at (2.5,2.5) {$\Sigma_{21}$};	
				\draw[red](0,0)--(1.5,1.5) ;
				\draw[-latex,red](-2.5,-2.5)--(-1.2,-1.2) ;
				\draw[red](-1.5,-1.5)--(0,0);
				\node[below] at (-2.5,-2.5) {$\Sigma_{24}$};	
				\draw[blue](0,0)--(-1.2,1.2) ;
				\draw[-latex,blue](-2.5,2.5)--(-1.2,1.2) ;
				\node[left] at (-2.5,2.5) {$\Sigma_{23}$};	
				\draw[-latex,blue](0,0)--(1.2,-1.2) ;
				\draw[blue](2.5,-2.5)--(1.2,-1.2) ;
				\node[right] at (2.5,-2.5) {$\Sigma_{22}$};	
				\draw(0,-0.2)node[below] {$0$};
				\draw(0.5,0.3) node[right]{$\Omega_{1}$};
				\draw(0,0.5) node[above]{$\Omega_{2}$};
				\draw(-0.5,0.3) node[left]{$\Omega_{3}$};
				\draw(-0.5,-0.3) node[left]{$\Omega_{4}$};
				\draw(0,-0.5) node[below]{$\Omega_{5}$};
				\draw(0.5,-0.3) node[right]{$\Omega_{6}$};
			\end{tikzpicture}
		\end{center}
		\caption{The contours $\Sigma_j$ and sectors $\Omega_j$ in the $\zeta$-plane defining RHP \textcolor{blue}{A.1}  }
\label{fgsd}
	\end{figure}

The solution of the  RH problem  \ref{PC}  can be given by
	\begin{equation}
		M^{pc}(\zeta,r)=I+\frac{1}{\zeta}
		\begin{pmatrix}
			0&-i\beta_{12} \\
			i\beta_{21} &0
		\end{pmatrix}+\mathcal{O}(\zeta^{-2}),
	\end{equation}
where   $\beta_{12}$ and $\beta_{21}$ are the complex constants
	\begin{align}\nonumber
		&\beta_{12}= \frac{-i\sqrt{2\pi}e^{i\pi/4}e^{-\pi\kappa/2}}{r_0\Gamma(-i\kappa)},\\
		&\beta_{21}= \frac{i\sqrt{2\pi}e^{-i\pi/4}e^{-\pi\kappa/2}}{r^*_0\Gamma(i\kappa)}=\frac{\kappa}{\beta_{12}}.
	\end{align}

	It can be shown that 

	\begin{lemma}\label{A1}    Let $\rho=||r||_{L^\infty(\mathbb{R})}$,  there exists a $C$ such that
		\begin{align}\label{eq:A.9}
			&|M^{pc}(\zeta)|\leq C,\quad \mbox{for all}\,\,\zeta\notin\mathbb{R},\\\label{eq:A.10}
			&\Big|M^{pc}(\zeta)-I-  \zeta^{-1}  {M^{pc}}_1  \Big|\leq C\rho|\zeta|^{-2}\quad \mbox{for }\,\,|\zeta|\ge1.
		\end{align}
	\end{lemma}

	\section{RH  problem  under  BT}\label{AB}

The following calculations are standard and   can be found in  \cite{24}.
Let $q (x)  \in H^{1,1}(\mathbb{R})$ be given and consider the associated ZS-AKNS operator and its reflection coefficient function $r(z)\in H^{1,1}(\mathbb{R})$.  Support that for each $x\in\mathbb{R}$, $2\times2$ matrix $\Phi(x,t;z)=\Psi(x,t;z)e^{it\theta(z)\sigma_3}$ solves the corresponding RH problem  with a finite number of simple bound states at
$$\mathcal{Z}=\{z\,|\,z=z_1,\dots,z_n\in\mathbb{C}^+\}, \ \mathcal{Z}^*=\{z|z=z^*_1,\dots,z^*_n\in\mathbb{C}_-\}.$$
	
	%\noindent\textbf{\expandafter{Riemann- Hilbert Problem B.1:} }
	\begin{problem}
		Find an analytic function $\Psi:\mathbb{C}\setminus(\mathbb{R}\cup\mathcal{Z}\cup\mathcal{Z}^*)\to SL_2(\mathbb{C})$ with the following properties
		\begin{enumerate}
			\item $ \Psi(x,z)e^{it\theta\sigma_3}=I+\mathcal{O}(z^{-1})$ as $z\to\infty$.
			\item $\Psi(x,z)e^{it\theta\sigma_3}$ takes continuous boundary values $\Psi_{\pm}(x,z)e^{it\theta\sigma_3}$ which satisfy the jump relation $\Psi_+(x,z)e^{it\theta\sigma_3}=\Psi_-(x,z)e^{it\theta\sigma_3}V(z)$ where
			\begin{equation}\label{B1}
				V(z)=\begin{pmatrix}
					1+|r(z)|^2&r^*(z)\\
					r(z)&1
				\end{pmatrix}.
			\end{equation}
			\item $\Psi(z)$ has simple poles at each $z_k\in\mathcal{Z}$ and $z_k^*\in\mathcal{Z^*}$ ($1\leq k\leq n$) at which
			\begin{equation}
				\underset{z=z_k}{\mbox{Res}} \Psi(x,z)e^{it\theta\sigma_3}=\underset{z\to z_k}{\mbox{lim}}\Psi(x,z)e^{it\theta\sigma_3}\begin{pmatrix}
					0&0\\
					c_k&0
				\end{pmatrix},
			\end{equation}
			\begin{equation}
				\underset{z=z_k^*}{\mbox{Res}} \Psi(x,z)e^{it\theta\sigma_3}=\underset{z\to z_k^*}{\mbox{lim}}\Psi(x,z)e^{it\theta\sigma_3}\begin{pmatrix}
					0&-c_k^*\\
					0&0
				\end{pmatrix}.
	\end{equation}\end{enumerate}\end{problem}
	
	The goal is to add in another simple bound state at $z=\xi\in\mathbb{C}^+\setminus\{z_1,\dots,z_n\}$ and simultaneously at $z=\xi^*\in\mathbb{C}^-\setminus\{z^*_1,\dots,z^*_n\}.$ We use a Darboux transformation $(z+P)(\partial_x-L)=(\partial_x-\overline{L})(z+P)$ where $P$ can chosen in the form $P=\mathfrak{b}(x)P_0\mathfrak{b}^{-1}(x)$ and $P_0$ is a constant matrix and $\mathfrak{b}=\mathfrak{b}(x)$ solves the equation $b'=Qb-i\frac{\sigma_3}{2}bP_0$, here the approprate choice here is $P_0=-\begin{pmatrix}
		\xi&0\\
		0&\overline{\xi}
	\end{pmatrix}$, where $\mathfrak{b}$ is determined below. Set
	\begin{equation}
		\tilde{\Psi}(x,z)=\mathfrak{b}(x)\mu(z)\mathfrak{b}^{-1}(x)\Psi(x,z)\mu^{-1}(z),
	\end{equation}
		where $\mu(z)=z+P_0=\begin{pmatrix}
		z-\xi&0\\
		0&z-\xi^*
	\end{pmatrix}$. Note that $\tilde{\Psi}(x,z)e^{it\theta\sigma_3}\to I$ as $z\to\infty$. Let $\tilde{c}(\xi)$ be any nonzero constant. We want to choose $\mathfrak{b}(x)$ so that $\tilde{\Psi}$ has a simple pole in the first column at $z=\xi$ and a simple pole in the second column at $z=\xi^*$ such that for $x\in\mathbb{R}$,
	\begin{equation}\nonumber
		\underset{z=\xi}{\mbox{Res}} \tilde{\Psi}e^{it\theta\sigma_3}=\underset{z\to \xi}{\mbox{lim}}\tilde{\Psi}e^{it\theta\sigma_3}\begin{pmatrix}
			0&0\\
			\tilde{c }(\xi)&0
		\end{pmatrix},
	\end{equation}
	\begin{equation}\nonumber
		\underset{z=\xi^*}{\mbox{Res}}\tilde{ \Psi}e^{it\theta\sigma_3}=\underset{z\to \xi^*}{\mbox{lim}}\tilde{\Psi}e^{it\theta\sigma_3}
		\begin{pmatrix}
			0&-\tilde{c^*(\xi)}\\
			0&0
		\end{pmatrix}.
	\end{equation}\
	Since
	\begin{equation}\nonumber
		\mathfrak{b}^{-1}\tilde{\Psi}=\begin{pmatrix}
			(\mathfrak{b}^{-1}\Psi)_{11}&(\mathfrak{b}^{-1}\Psi)_{12}\frac{z-\xi}{z-\xi^*}\\
			\\
			(\mathfrak{b}^{-1}\Psi)_{21}\frac{z-\xi^*}{z-\xi}
			&(\mathfrak{b}^{-1}\Psi)_{21}\end{pmatrix},
	\end{equation}
	we have
	\begin{equation}\nonumber
		\underset{z_\xi}{\mbox{Res}}\,\mathfrak{b}^{-1}\tilde{ \Psi}=\begin{pmatrix}
			0&0\\
			(\mathfrak{b}^{-1}\Psi)_{21}(x,\xi)(\xi-\xi^*)&0
		\end{pmatrix},
	\end{equation}\
	but
	\begin{equation}\nonumber
		\underset{z\to\xi}{\mbox{lim}}\,\mathfrak{b}^{-1}\tilde{\Psi}\begin{pmatrix}
			0&0\\
			\tilde{c}(\xi)&0
		\end{pmatrix}
		=\begin{pmatrix}
			0&0\\
			\tilde{c}(\xi)(\mathfrak{b}^{-1}\Psi)_{22}(x,\xi)&0
		\end{pmatrix},
	\end{equation}
	and hence we must have
	\begin{equation}\nonumber
		(\xi-\xi^*)(e_2,\mathfrak{b}^{-1}\Psi(x,\xi)e_1)=\tilde{c}(\xi)(e_2,\mathfrak{b}^{-1}\Psi(x,\xi)e_2).
	\end{equation}
	Therefore, it follows necessarily that
	\begin{equation}\nonumber
		\mathfrak{b}(x)e_1=c_1(x)\big(\Psi(x,\xi)e_1-\frac{\tilde{c}(\xi)}{\xi-\xi^*}\Psi(x,\xi)e_2\big),
	\end{equation}
	for some nonzero function $c_1(x)$. Similarly for $z=\xi^*$, we see that
	\begin{equation}\nonumber
		\mathfrak{b}(x)e_2=c_2(x)\big(-\frac{\tilde{c}^*(\xi)}{\xi-\xi^*}\Psi(x,\xi^*)e_1+\Psi(x,\xi^*)e_2\big),
	\end{equation}
	for some nonzero function $c_2(x)$, Observe that $c_1(x)$ and $c_2(x)$ factor out in the formula (B.1) for $\tilde{\Psi}$. Set
	\begin{equation}
		\mathfrak{b}=\big(\Psi(x,\xi)\begin{pmatrix}
			1\\
			-\frac{\tilde{c}(\xi)}{\xi-\xi^*}\\
		\end{pmatrix}\quad \Psi(x,\xi^*)\begin{pmatrix}
			-\frac{\tilde{c}^*(\xi)}{\xi-\xi^*}\\
			1
		\end{pmatrix}\big).
	\end{equation}
	From the symmetry we see that $\mathfrak{b}_2=\begin{pmatrix}
		0&1\\
		-1&0
	\end{pmatrix}\mathfrak{b}_1^*$, where $\mathfrak{b}=(\mathfrak{b}_1,\mathfrak{b}_2)$. Thus, $\mbox{det} \mathfrak{b}(x)=|\mathfrak{b}_1)_1(x)|^2+|(\mathfrak{b}_1)_2(x)|^2>0$ and hence $\mathfrak{b}(x)$ is invertible for all $x\in\mathbb{R}$. The jump matrix $\tilde{V}$ for $\Psi(x,z)$ is given by
	\begin{align}
		\tilde{V}(z)&=\tilde{\Psi}_-^{-1}(x,z)\tilde{\Psi}_+(x,z)=\mu(z)V(z)\mu^{-1}(z)\\
		&=\begin{pmatrix}
			1+|\tilde{r}(z)|^2&\tilde{r}^*(z)\\
			\tilde{r}(z)&1
		\end{pmatrix},\qquad z\in\mathbb{R},
	\end{align}
	where
	$$\tilde{r}(z)=r(z)\frac{z-\xi}{z-\xi^*}.$$
	A straightforward calculation shows that for $1\leq k\leq n$,
	\begin{align}\nonumber
		&\underset{z_k}{\mbox{Res}} \tilde{\Psi}=\underset{z\to z_k}{\mbox{lim}}\tilde{\Psi}\begin{pmatrix}
			0&0\\
			\tilde{c}_k(z_k)&0
		\end{pmatrix},\\
		\,\,\,
		&\underset{z_k^*}{\mbox{Res}} \tilde{\Psi}=\underset{z\to z_k^*}{\mbox{lim}}\tilde{\Psi}\begin{pmatrix}
			0&-\tilde{c}_k(z_k)^*\\
			0&0
		\end{pmatrix},
	\end{align}
	where
	$$\tilde{c}(z_k)=c(z_k)\frac{z_k-\xi^*}{z_k-\xi}.$$
	The above calculations show that  $\tilde{M}(x,z)=\tilde{\Psi}(x,z)e^{-it\theta\sigma_3}$
	solves the RH Problem \ref{RH21}. Note that
	$$\tilde{a}(z)=\frac{z-\xi}{z-\xi^*}a(z),$$
	where $z(z)$ and $\tilde{a}(z)$ are the scattering functions for $\Psi(x,z)$ and $\tilde{\Psi}(x,z)$, respectively. From the fact that $\Psi_x+iz[\sigma_3,\Psi]=Q\Psi,Q=\begin{pmatrix}
		0&q\\
		-q^*&0
	\end{pmatrix}$, we have
	$$Q=-i[\sigma_3,\Phi_1],\quad \Psi=I+\frac{\Phi_1}{z}+\mathcal{O}(z^{-1}),$$
	as $z\to\infty$ in any cone $|\mbox{Im} z|>c|\mbox{Re} z|$, and $c>0$. Let $\mu_1\begin{pmatrix}
		\xi&0\\
		0&\xi^*
	\end{pmatrix}.$ For $\tilde{\Psi}(x,z)=\tilde{\Phi}(x,z)e^{-it\theta\sigma_3}$
	\begin{align}
		\tilde{\Phi}&=\mathfrak{b}\big(I-\frac{\mu_1}{z}\big)\mathfrak{b}^{-1}\big(I+\frac{\Phi_1}{z}+\mathcal{O}(z^{-1})\big)\big(I-\frac{\mu_1}{z}\big)^{-1}\\
		&=I+\frac{\Psi_1-\mathfrak{b}\mu_1\mathfrak{b}^{-1}+\mu_1}{z}+\mathcal{O}(z^{-1}),
	\end{align}
	and hence
	\begin{align}
		\tilde{q}(x)&=-i[\sigma_2,\Phi_1-\mathfrak{b}\mu_1\mathfrak{b}^{-1}+\mu_1]_{12}\\
		&=q(x)+i(\xi-\xi^*)\frac{(\mathfrak{b}_1)_1(\mathfrak{b}_1)_2^*}{|(\mathfrak{b}_1)_1|^2+|(\mathfrak{b}_1)_2|^2}.
	\end{align}
	One can also Darboux transformations similar to (\ref{B1}) to remove eigenvalues. We do not provide any further details, except to not that at each step, if the poles at $z=z_k,z_k^*$ are removed, then $r(z)\to\tilde{r}(z)\frac{z-z_k^*}{z-z_k}$, etc. \vspace{8mm}

	\noindent\textbf{Acknowledgements}
	
	This work is supported by  the National Natural Science
	Foundation of China (Grant No. 12271104,  51879045).\vspace{2mm}
	
	\noindent\textbf{Data Availability Statements}
	
	The data that supports the findings of this study are available within the article.\vspace{2mm}
	
	\noindent{\bf Conflict of Interest}
	
	The authors have no conflicts to disclose.

\end{document}